\documentclass[11pt, oneside, a4paper]{scrartcl}

\usepackage[english]{babel}       		
\usepackage{amsmath, amssymb, amsthm,xcolor,mathrsfs } 	
\definecolor{mycolor1}{rgb}{0.00000,0.44700,0.74100}%
\definecolor{ccgray}{gray}{0.1}
\usepackage{authblk} 
\usepackage{hyperref,graphicx}
\usepackage{nomencl}
\makenomenclature
\usepackage{mathdots}
\usepackage{multirow}

\usepackage{booktabs}  

\usepackage{subcaption}
\usepackage{bm}   

\theoremstyle{plain}
\newtheorem{theorem}{Theorem}[] 
\newtheorem{proposition}[theorem]{Proposition}

\newtheorem{lemma}[theorem]{Lemma}
\theoremstyle{definition}

\newtheorem{example}{Example}

\newtheorem{algo2}{Algorithm}
\theoremstyle{remark}

\newtheorem{remark}{\textbf{Remark}}

\theoremstyle{plain}
\newtheorem{assumption}{Assumption}

\usepackage[square,numbers]{natbib}        
\usepackage[ruled,vlined]{algorithm2e}
\usepackage[font={small,it}, labelfont=bf]{caption}[2004/07/16] 
\usepackage{xparse, tikz}
\usepackage{pgfplots}
\pgfplotsset{compat=1.12}
\usetikzlibrary{backgrounds}
\usetikzlibrary{arrows,calc,shapes}
\usetikzlibrary{fit,positioning,decorations.pathreplacing}
\pgfplotsset{every tick label/.append style={font=\tiny}}
\usepackage{mathtools}

\DeclareMathOperator{\bK}{\bm{{\mathrm{K}}}}
\newcommand*{\No}{\mathcal{N}} 
\newcommand*{\Mo}{\mathcal{D}} 

\renewcommand{\phi}{\varphi}		
\renewcommand{\theta}{\vartheta}
\renewcommand{\epsilon}{\varepsilon}
\newcommand{\field}[1]{\mathbb{#1}}
\newcommand{\R}{\field{R}}
\newcommand{\Z}{\field{Z}}
\newcommand{\N}{\field{N}}

\newcommand{\C}{\field{C}}
\usepackage{relsize}
\newcommand{\transpose}{\mathsmaller {\mathrm{T}}}
\usepackage{nicefrac}  

\title{\LARGE Spectrum-based stability analysis and stabilization \\ of a class of time-periodic time delay systems}
\date{}

\author[1]{{\large Wim Michiels}\thanks{Wim.Michiels@cs.kuleuven.be}}
\author[1]{{\large Luca Fenzi}\thanks{Luca.Fenzi@cs.kuleuven.be}}
\affil[1]{{\footnotesize Department of Computer Science, KU Leuven, Belgium}}
\usepackage{dsfont}
\usepackage{enumitem}
\usepackage{physics}
\newcommand*{\ie}{\textit{i.e.}\@\xspace}  

\begin{document}

\maketitle
\vspace{-0.1\textheight}

\begin{abstract}
We develop an eigenvalue-based approach for the stability assessment and stabilization of linear systems with multiple delays and periodic coefficient matrices. Delays and period are assumed commensurate numbers, such that the Floquet multipliers can be characterized as eigenvalues of the monodromy operator and by the solutions of a finite-dimensional non-linear eigenvalue problem, where the evaluation of the characteristic matrix involves solving an initial value problem.
We demonstrate that such a dual interpretation can be exploited in a two-stage approach for computing dominant Floquet multipliers, where global approximation is combined with local corrections. Correspondingly, we also  propose two novel characterizations of left eigenvectors. Finally, from the nonlinear eigenvalue problem formulation, we derive computationally tractable expressions for derivatives of Floquet multipliers with respect to parameters, which are beneficial in the context of stability optimization. Several numerical examples show the efficacy and applicability of the presented results.
\end{abstract}


\section{Introduction}
Linear time-periodic delay differential equations appear in mathematical models for a variety of dynamical systems including  machining processes such as milling, they describe periodic control strategies such as act-and-wait control in the presence of feedback delays, and they appear in the local stability analysis of periodic solutions of systems governed by nonlinear delay differential equations. We refer to \cite{Szalai2006,Insperger2011,Breda2015} and reference therein for an overview.

In this paper we consider linear time-periodic systems with multiple delays, described by model 
\begin{equation}
\dot{x}(t)=\sum_{j=0}^hA_j(t)x(t-\tau_j),
\label{eq:system}
\end{equation}
where $h\in\N$, state variable $x(t)\in\R^d$, and functions
$
A_j:\ \R\to\R^{d\times d},\ \    t \mapsto A_j(t)
$
are smooth and $T$-periodic, for $j=0,\ldots,h$. The period satisfies $T>0$ and the delays are sorted in increasing order such that 
$
0=\tau_0\leq \tau_1 <\tau_2<\cdots<\tau_h. 
$
Throughout the paper we address the case where the time-delays and period are commensurate.  More precisely, we make the following assumption.
\begin{assumption}  \label{as1} There exist  real number $\Delta>0$, integers $N$ and $n_j$, for $j=1,\ldots,h$ such that the period and delays satisfy
	\begin{equation*}
	T=N\Delta, \quad \tau_j=n_j\Delta,\ \ j=1,\ldots,h.
	\end{equation*}
\end{assumption}
The stability of the zero solution of \eqref{eq:system} can be inferred 
from the location of the Floquet multipliers, which are eigenvalues of the monodromy operator. 
As we shall see, under Assumption~\ref{as1} the Floquet multipliers can  alternatively be characterized by the solutions of a finite-dimensional nonlinear eigenvalue problem, where the evaluation of the product of the so-called characteristic matrix with a vector involves the solution of an initial value problem for an ordinary differential equation; see \cite{Voss2013,Guettel2017}  for an overview of properties and methods for generic nonlinear eigenvalue problems, and   \cite{Sieber2011,Szalai2006} and the references therein for characteristic matrices of periodic time-delay systems.  This result generalizes Theorem~2.1 of \cite{Skubachevskii2006} where scalar systems with one delay are considered. We also refer to \cite{Jarlebring2018,Rott2010},  where the one-delay case with $T=\tau_1$ is considered and the nonlinear eigenvalue problem is expressed in terms of  Floquet exponents, rather than Floquet multipliers.

Several methods have been proposed in the literature to approximate Floquet multipliers by discretizing the monodromy operator. 
The semi-discretization approach \cite{Insperger2011} is based on approximating the delays by sawtooth functions, in such a way that the stability of the approximate system can be inferred from the stability a  higher-order differential equation in discrete time.  Another frequently used discretization technique consists of approximating the monodromy operator using (pseudo)spectral collocation, variants  of which have been proposed in \cite{Bueler2007,Butcher2011,Engelborghs2001,Sieber2014} and in \cite[Chapter~6]{Breda2015}. 
  In this paper, a spectral discretization based on spline collocation is adopted, so  that the approximation of the eigenvalues is consistent with solving the boundary value problem, induced by the nonlinear eigenvalue problem formulation.  This results in a standard or polynomial  eigenvalue problem for which both a direct solver, and, for large-scale problems, an Arnoldi type algorithm are proposed.

At the same time, the nonlinear eigenvalue problem formulation or, equivalently,
 the characteristic matrix formulation, allows to compute the Floquet multipliers by applying iterative solvers for systems of nonlinear  equations, see \cite{Jarlebring2018} and the references therein.  As a first main contribution, this leads us to  a novel two-stage approach for efficiently and reliably computing dominant Floquet multipliers. First,  multiple Floquet multiplier approximations are obtained by discretizing the monodromy operator and solving the resulting eigenvalue problem. The advantage lies in the fact that multiple Floquet multiplier approximations are found, yielding a more than local picture about the distribution of the spectrum. Subsequently, the eigenvalues and eigenvectors are used 
 as starting values for Broyden's method applied to the nonlinear eigenvalue problem formulation.  The latter allows us to significantly increase accuracy of individual Floquet multiplier approximations and at the same time  remove spurious eigenvalues.  Conceptually similar to the stability assessment of equilibria of autonomous linear time-delay system  in \cite{Michiels2011} and in the package DDE-BIFTOOL \cite{Sieber2014}, this two-stage approach of global approximation and local correction  exploits the dual interpretation of the Floquet multipliers, arising from either an infinite-dimensional linear or a finite-dimensional nonlinear eigenvalue problem.
Finally, it should be pointed out that the approach of locally solving a nonlinear system of equations derived from the characteristic matrix has also been successfully used to compute bifurcations of periodic solutions of parametrized nonlinear time-periodic systems in article \cite{Szalai2006}, which is from a methodological point of view closely related to the software package PDDE-CONT.

As a second main contribution, we provide a dual characterization of left eigenvectors of the characteristic matrix, in terms of right eigenvectors of a  ``transposed'' nonlinear eigenvalue problem, and in terms of right eigenfunctions of the monodromy operator corresponding to a ``transposed''  periodic time-delay system of \eqref{eq:system},
whose construction involves both taking the transpose and a linear transformation of the argument of coefficient matrices $A_j(t)$, $j=0,\ldots,h$.  This allows us to trivially  extend the presented methods for computing right eigenvectors to left eigenvectors. It also avoids the construction of the full characteristic matrix, in order to compute a left eigenvector corresponding to an already computed  Floquet multiplier.

As a third main contribution, we show  that the nonlinear eigenvalue problem formulation allows to characterize and compute  derivatives of  simple Floquet multipliers  with respect to parameters, by solving  a related variational equation.  
This approach has potential to the initialization of the deflated Broyden's method in \cite{Jarlebring2018}.   As we shall illustrate, it is particularly beneficial in the context of stability optimization and for the design of  stabilizing controllers, as it can be employed as basis of an algorithm for minimizing the spectral radius of the monodromy operator. 
Existing approaches include  \cite{Sheng2005,Nazari2014}, where the monodromy operator is discretized  by semi-discretization and  by  spectral collocation, respectively (see also \cite{Niu2015} for pole placement based on a discretization of the monodromy operator by a Runge-Kutta time-stepping scheme). 
The proposed technique improves these methods in two ways.  First, the computation of both objective function and derivatives is at the level of the nonlinear eigenvalue problem, serving for local corrections in the aforementioned two-stage approach, while the existing techniques are based on discretizing the (parametrized) system in a preliminary step. Second, the reliability and efficiency of the optimization process is improved by employing derivatives of the objective function.

\smallskip

The structure of this paper is as follows. In section~\ref{sec:diffeq}  the spectral properties of system \eqref{eq:system} are addressed, focusing on the characterization of Floquet multipliers in terms of a nonlinear eigenvalue problem.  In section~\ref{sec:Stabass} algorithms for computing Floquet multipliers and their right eigenvector are presented, which are tightly linked to the dual interpretation of Floquet multipliers.   Characterizations of left eigenvectors are presented in  section~\ref{sec:left}.  Section~\ref{sec:deFloMustab} is devoted to the sensitivity analysis of the Floquet multipliers, with application to the design of stabilizing controllers. Section~\ref{sec:examples} illustrates the efficacy of the novel approach on numerical experiments. The conclusion are presented in section~\ref{sec:conclusion}.

\smallskip 

\noindent\textbf{Notations.\ } We denote by $\otimes$, the Kronecker product, by  $ k \bmod_N$,  the remainder after the division of $k$ by $N$; by  $\lfloor \nicefrac{k}{N} \rfloor$ the rounding to the largest integer, which does not exceed   $\nicefrac{k}{N}$; and by $\lceil \nicefrac{k}{N} \rceil$ the rounding to the smallest integer which is not less than $\nicefrac{k}{N}$.  The transpose and the  complex conjugate transpose of $p$ are respectively denoted by $p^\transpose$, and  $p^*$. $\bar{\nu}$ denotes the complex conjugate of $\nu\in\C$.

\section{Spectral properties and stability}\label{sec:diffeq}
In section~\ref{sec:preliminary} we briefly review, based on \cite[Chapter~8]{Hale1977}, the Floquet theory for linear periodic delay differential equations, which infers  the stability properties of system \eqref{eq:system} from the solutions of an infinite-dimensional linear eigenvalue problem corresponding to the monodromy operator. In section~\ref{sec:nleig} we show that the eigenvalues of this operator can alternatively be obtained from a finite-dimensional nonlinear eigenvalue problem.

\subsection{Preliminaries: the infinite-dimensional linear eigenvalue problem}\label{sec:preliminary}
In order to define a solution of \eqref{eq:system} it is not sufficient to specify $x$ at the starting time yet in general a function segment over a time-interval of length $\tau_h$ is needed.  More precisely, for any initial function $\phi\in X$, where $X=\mathcal{C}([-\tau_h,\,0],\C^d)$ and $t_0\in\R$, the initial value problem
\begin{equation}
\begin{cases}
\dot{x}(t)=\sum_{j=0}^hA_j(t)x(t-\tau_j), &t\in[t_0,\,\infty)\\
x(t)=\phi(t-t_0),\quad &t\in[t_0-\tau_h,\,t_0]
\end{cases} 
\label{eq:ivp}
\end{equation}
has a unique forward solution, which we denote by $x(t;t_0,\phi)$. The corresponding state at time $t$, $t\geq t_0$, i.e. the minimal information to continue the solution is denoted by $x_t(\cdot;t_0,\phi)\in X$, defined by
\[
x_t(\theta;t_0,\phi)=x(t+\theta;t_0,\phi),\ \ \theta\in [-\tau_h,\, 0].
\]
The translation along the solutions is described by the solution operator 
$\mathcal{T}(t_1, t_0): X\to X$, parametrized by $t_0,\, t_1\in\R,\, t_1\geq 0$  and defined through the relation
\[
\mathcal{T}(t_1, t_0)\; \phi= x_{t_0+t_1}(\cdot;t_0,\phi),\ \ \phi\in X.
\]

It can be shown that the spectrum of operator $\mathcal{T}(T,t_0)$, with $T$ the period of functions $A_j$, is
an at most countable compact set in the complex plane with zero as only possible accumulation point. The spectrum is independent of the choice of $t_0$ and all its nonzero elements are eigenvalues.  Operator $\mathcal{T}(T,0)$ is called the monodromy operator and denoted by $\mathscr{U}$ in what follows. Hence, we have
\[
\mathscr{U}\phi=x_T(\cdot;0,\phi), \quad \phi\in X.
\]
The nonzero eigenvalues of the monodromy operator are called Floquet multipliers. By definition they satisfy  the infinite-dimensional linear eigenvalue problem 
\begin{equation}
\mathscr{U}\phi=\mu\; \phi, \qquad \mu\in\C, \ \phi\in X\setminus\{0\}.
\label{eq:infeig}
\end{equation}
As the Floquet multipliers determine the growth/decay of solutions of \eqref{eq:ivp} in time-interval of length $T$ and the system is $T$-periodic, they are important for stability assessment.  In particular, the zero solution of \eqref{eq:system} is asymptotically stable if and only if all Floquet multipliers have modulus strictly smaller than one. This motives the developments of algorithms for computing dominant Floquet multipliers.

A graphical interpretation of the action of the monodromy operator and its associated  eigenvalue problem is given in Figure~\ref{fig:monodromy}.

\begin{figure}[h]
	\begin{subfigure}[h]{0.47\linewidth}
		\centering
	\resizebox{\linewidth}{!}{	
		\begin{tikzpicture}[
		thick,
		>=stealth',
		dot/.style = {
			draw,
			fill = white,
			circle,
			inner sep = 0pt,
			minimum size = 4pt
		}
		]
		
		\useasboundingbox (-2.6,-0.6) rectangle (5.2,3.2); 
		\coordinate (O) at (0,0);
		\draw[-] (-2.3,0) -- (5,0) coordinate[label = {}] (x);
		\draw[-] (-2.3,0.15) --  (-2.3,-0.15) coordinate[label = {below: $-\tau_h$}] (xmax);
		\draw[-] (2.7,0.1) --  (2.7,-0.1) coordinate[label = {below: $T-\tau_h$}] (xmax);
		\draw[-] (5,0.15) --  (5,-0.15) coordinate[label = {below: $T$}] (T);
		\draw[->] (0,-0.15) coordinate[label = {below: $0$}] (y) -- (0,2.2) ;
		\node[draw=mycolor1] at (-1.7,3) {\textcolor{black}{$T>\tau_h$}};
		
		\draw[samples=100,domain=-2.3:0] plot (\x,{1.1+0.5*sin(2.3*deg(\x))});
		\draw[decorate,decoration={brace,amplitude=5pt},gray] (-2.3,1.7) -- (0,1.7)  node[black,midway,yshift=+0.35cm] (q) {\textcolor{gray}{$\phi$}};
		\draw[-,gray] (-2.3,0.15) --  (-2.3,1.71);

		\draw[black!30!gray,dashed] plot[smooth] coordinates {(0,1.1) (1,2) (2,1.6) (2.7,{1.5*(1.1+0.5*sin(2.3*deg(2.7-5)))})};
		
		\draw[samples=100,domain=2.7:5] plot (\x,{1.5*(1.1+0.5*sin(2.3*deg(\x-5)))});
		\draw[decorate,decoration={brace,amplitude=5pt},gray] (2.7,2.4) -- (5,2.4)  node[black,midway,yshift=+0.35cm] (uq) {\textcolor{gray}{$\mathscr{U}\phi$}};
		\draw[-,gray] (5,0.15) --  (5,2.41);
		\draw[-,gray] (2.7,0.1) --  (2.7,2.41);
		\draw[->] (q)  to[bend left=10]  (uq)  ;
		\node[] at (1.6,3){\textcolor{black}{$\mathscr{U}\phi=\mu \phi$}};   
		\end{tikzpicture}}
	\end{subfigure}\ 
	\begin{subfigure}[h]{0.47\linewidth}
		\centering
		\resizebox{\linewidth}{!}{	
			\begin{tikzpicture}[
			thick,
			>=stealth',
			dot/.style = {
				draw,
				fill = white,
				circle,
				inner sep = 0pt,
				minimum size = 4pt
			}
			]
			
			\useasboundingbox (-4.3,-0.6) rectangle (3.3,3.2); 
			\coordinate (O) at (0,0);
			\draw[-] (-4,0) -- (3,0) coordinate[label = {}] (x);
			\draw[-] (-4,0.15) --  (-4,-0.15) coordinate[label = {below: $-\tau_h$}] (xmax);
			\draw[-] (-1,0.1) --  (-1,-0.1) coordinate[label = {below: \textcolor{black}{$T-\tau_h$}}] (xmax);
			\draw[-] (3,0.15) --  (3,-0.15) coordinate[label = {below: $T$}] (T);
			\draw[->] (0,-0.15) coordinate[label = {below: $0$}] (y) -- (0,2.2) ;
			\node[draw=mycolor1] at (-3.4,3) {\textcolor{black}{$T<\tau_h$}};

			\draw[samples=100,domain=-4:0] plot (\x,{1.1+0.5*sin(2*deg(\x-4))});
			\draw[decorate,decoration={brace,amplitude=5pt},gray] (-4,1.7) -- (0,1.7) 
			node[black,midway,yshift=+0.35cm] (q) {\textcolor{gray}{$\phi$}};
			\draw[-,gray] (-4,0.15) --  (-4,1.71);

			\draw[samples=100,domain=-0:3] plot (\x,{0.9*(1.12+0.5*sin(2*deg(\x-1)))});
			\draw[-,gray] (3,0.15) --  (3,2.21);
			\draw[-,gray] (-1,2.21) --  (-1,0.1);
			\draw[decorate,decoration={brace,amplitude=5pt},gray] (-1,2.2) -- (3,2.2)  node[black,midway,yshift=+0.35cm] (uq) {\textcolor{gray}{$\mathscr{U}\phi$}};

			\draw[->] (q)  to[bend left=10]  (uq)  ;
			\node[] at (-0.9,2.8){\textcolor{black}{$\mathscr{U}\phi=\mu \phi$}};
			\end{tikzpicture}
			}
	\end{subfigure}
	\caption{The monodromy operator translates function $\phi\in\mathcal{C}([-\tau_h,\,0],\R^d)$ along the corresponding solution over a time-interval of length  $T$. In both cases the depicted initial function $\phi$ corresponds to an eigenfunction of $\mathscr{U}$. }\label{fig:monodromy}
\end{figure}
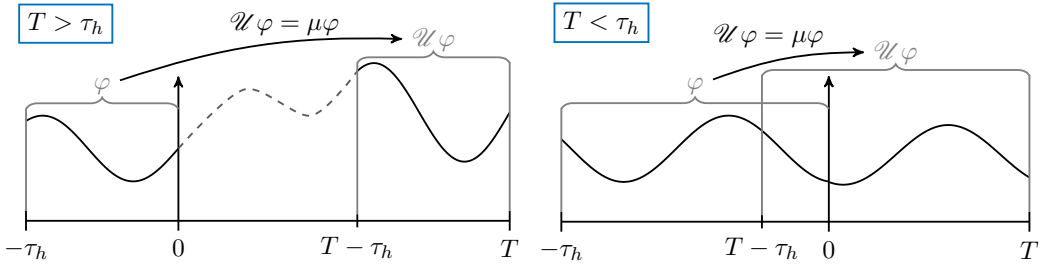	

\subsection{The finite-dimensional nonlinear eigenvalue problem}\label{sec:nleig}

We show that the Floquet multipliers can not only be obtained by solving infinite-dimensional linear eigenvalue problem \eqref{eq:infeig}, but also
from the solutions of a finite-dimensional nonlinear eigenvalue problem.
This dual interpretation plays a major role in the subsequent developments of the paper.
We start with a technical lemma.
\begin{lemma}	\label{th:operator}
	Number $\mu\in\C\setminus\{0\}$ is a Floquet multipliers of system \eqref{eq:system} if and only if  there exists a continuous $\C^{Nd}$-valued function, $\bm{q}(s)$ with  $\bm{q}(0)\neq0$,  which satisfies the boundary value problem
	\begin{equation}
	\begin{cases}
	\dot{\bm{q}}(s)=A(s,\mu)\bm{q}(s),\quad s\in[0,\,1],\\
	\bm{q}(1)=B(\mu)\bm{q}(0),
	\end{cases}\label{eq:difeq}
	\end{equation}
	where
	\[
	B(\mu)=\begin{pmatrix} 0 & I_N \\
			\mu & 0 \end{pmatrix} \otimes I_d,
	\]
	and the differential is described by
	\begin{equation}
	\dot{q}_n(s)=\Delta\sum_{j=0}^hA_j\left((s+n-1)\Delta\right)\mu^{a_{n-n_j}} q_{b_{n-n_j}}(s),\ \ \  n=1,\ldots,N,
	\label{eq:Asmu}
	\end{equation}
	with 
	\[
	a_{k}=\left\lfloor \frac{k-1}{N}\right\rfloor, \quad b_{k}=\left(k-1\right)\bmod_N+1,
	\]
	such that $\bm{q}(s)=\big(q_1^\transpose(s)\,\cdots\,q_N^\transpose(s)\big)^\transpose$.
\end{lemma}

A detailed proof is given in Appendix~\ref{ap1}. To sketch the idea behind one of the implications, consider the simplest case $h=N=n_1=\Delta=1$. Then a solution~$x$, emanating at $t=0$ from an eigenfunction $\phi$ corresponding to Floquet multiplier $\mu$, satisfies
\[
\begin{array}{lll}
\dot x(t;0,\phi) &=& A_0(t) x(t;0,\phi)+ A_1(t) \phi(t)\\
&=& A_0(t) x(t;0,\phi)+ A_1(t) \frac{x(t;0,\phi)}{\mu}
\end{array}
\]
for $t\in [0,\, 1]$, as well as $x(1;0,\phi)=\mu x(0;0,\phi)$. Setting $\bm{q}(t)=x(t;0,\phi)$ yields equations of the form \eqref{eq:difeq}.   A graphical illustration of the generalization is given in Figure~\ref{fig:th}. Variable $s$ in \eqref{eq:difeq} is a local coordinate inside each interval of length $\Delta$. 
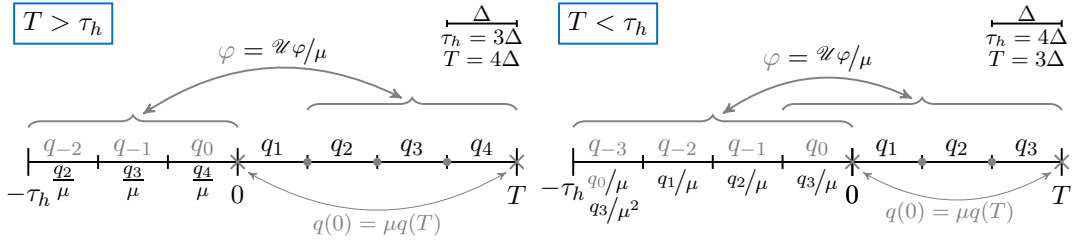
\begin{figure}[h]
	\begin{subfigure}[h]{0.47\linewidth}
		\centering
		\resizebox{\linewidth}{!}{	\begin{tikzpicture}[
			thick,
			>=stealth',
			dot/.style = {
				draw,
				fill = white,
				circle,
				inner sep = 0pt,
				minimum size = 4pt
			}
			]
			
			\useasboundingbox (-3.2,-1.1) rectangle (4.2,2.2); 
			\coordinate (O) at (0,0);
			\draw[-] (-3,0) -- (4,0) coordinate[label = {}] (x);
			\draw[-] (-3,0.2) --  (-3,-0.2) coordinate[label = {below: $-\tau_h$}] (xmax);
			\draw[-] (4,0.2) --  (4,-0.2) coordinate[label = {below: $T$}] (T);

			\node[draw=mycolor1] at (-2.5,2) {\textcolor{black}{$T>\tau_h$}};
			
			\draw[-] (3,2) --  (4,2) coordinate[label = {}] ();
			\draw[-] (3,2.05) --  (3,1.95) coordinate[label = {}] ();
			\draw[-] (4,2.05) --  (4,1.95) coordinate[label = {}] ();
			\node[]  at (3.5,2.15){{\footnotesize{$\Delta$}}};
			\node[]  at (3.5,1.5){{\footnotesize{$T=4\Delta$}}};
			\node[]  at (3.5,1.8){{\footnotesize{$\tau_h=3\Delta$}}};

			\draw[-] (1,0.1) --  (1,-0.1);
			\draw[-] (2,0.1) --  (2,-0.1);
			\draw[-] (3,0.1) --  (3,-0.1);
			
			\draw[-] (-1,0.1) --  (-1,-0.1);
			\draw[-] (-2,0.1) --  (-2,-0.1);
			
			\node[] (q1) at (0.5,0.2){\textcolor{black}{$q_1$}};
			\node[] (q2) at (1.5,0.2){\textcolor{black}{$q_2$}};
			\node[] (q3) at (2.5,0.2){\textcolor{black}{$q_3$}};
			\node[] (q4) at (3.5,0.2){\textcolor{black}{$q_4$}};
			
			\node[] at (-0.5,0.2){\textcolor{gray}{${q_0}$}};
			\node[] at (-1.5,0.2){\textcolor{gray}{${q_{-1}}$}};
			\node[] at (-2.5,0.2){\textcolor{gray}{${q_{-2}}$}};

			\node[] (q0) at (-0.5,-0.3){\textcolor{black}{$\frac{q_4}{\mu}$}};
			\node[] (q-1) at (-1.5,-0.3){\textcolor{black}{$\frac{q_3}{\mu}$}};
			\node[] (q-2) at (-2.5,-0.3){\textcolor{black}{$\frac{q_2}{\mu}$}};
			
			\draw[-] (0,0.2) --  (0,-0.2) coordinate[label = {below: $0$}] (T);
			
			\draw[decorate,decoration={brace,amplitude=5pt},gray] (-3,0.5) -- (0,0.5)  node[black,midway,yshift=+0.1cm] (q) {}; 
			
			\draw[decorate,decoration={brace,amplitude=5pt},gray] (1,0.75) -- (4,0.75)  node[black,midway,yshift=+0.1cm] (uq) {};
			\draw[<->,gray] (q)  to[bend left]  (uq) ;
			\node[] at (0.5,1.6){$\textcolor{gray}{\phi}=\nicefrac{\mathscr{U}\phi}{\mu}$};

			\node[circle, fill, minimum size=2,inner sep=1.2pt, outer sep=1pt, gray] at (1,0){};
			\node[circle, fill, minimum size=2,inner sep=1.2pt, outer sep=1pt, gray] at (2,0){};
			\node[circle, fill, minimum size=2,inner sep=1.2pt, outer sep=1pt, gray] at (3,0){};
			
			\draw[-,gray] (0.1,0.1) --  (-0.1,-0.1);
			\draw[-,gray] (0.1,-0.1) --  (-0.1,0.1);
			\draw[-,gray] (4.1,0.1) --  (3.9,-0.1);
			\draw[-,gray] (4.1,-0.1) --  (3.9,0.1);    
			
			\draw[<->,thin,gray] (0.15,-0.15)  to[bend right]  (3.85,-0.15);
			\node[] at (2,-0.9){{\footnotesize\textcolor{gray}{$q(0)=\mu q(T)$}}};
			\end{tikzpicture}}
	\end{subfigure}\ \ 
	\begin{subfigure}[h]{0.47\linewidth}
		\centering
		\resizebox{\linewidth}{!}{	\begin{tikzpicture}[
			thick,
			>=stealth',
			dot/.style = {
				draw,
				fill = white,
				circle,
				inner sep = 0pt,
				minimum size = 4pt
			}
			]
			
			\useasboundingbox (-4.2,-1.1) rectangle (3.2,2.2); 
			\coordinate (O) at (0,0);
			\draw[-] (-4,0) -- (3,0) coordinate[label = {}] (x);
			\draw[-] (-4,0.2) --  (-4,-0.2);
			\node[] at (-4.15,-0.4){\textcolor{black}{$-\tau_h$}}; 
			
			\draw[-] (3,0.2) --  (3,-0.2) coordinate[label = {below: $T$}] (T);
			\draw[-] (0,0.2) --  (0,-0.2) coordinate[label = {below: $0$}] (T);

			\node[draw=mycolor1] at (-3.5,2) {\textcolor{black}{$T<\tau_h$}};
			\draw[-] (2,2) --  (3,2) coordinate[label = {}] ();
			\draw[-] (2,2.05) --  (2,1.95) coordinate[label = {}] ();
			\draw[-] (3,2.05) --  (3,1.95) coordinate[label = {}] ();
			\node[]  at (2.5,2.15){{\footnotesize{$\Delta$}}};
			\node[]  at (2.5,1.5){{\footnotesize{$T=3\Delta$}}};
			\node[]  at (2.5,1.8){{\footnotesize{$\tau_h=4\Delta$}}};
			
			\draw[-] (1,0.1) --  (1,-0.1);
			\draw[-] (2,0.1) --  (2,-0.1);

			\draw[-] (-1,0.1) --  (-1,-0.1);
			\draw[-] (-2,0.1) --  (-2,-0.1);
			\draw[-] (-3,0.1) --  (-3,-0.1);
			\node[] (q1) at (0.5,0.2){\textcolor{black}{$q_1$}};
			\node[] (q2) at (1.5,0.2){\textcolor{black}{$q_2$}};
			\node[] (q3) at (2.5,0.2){\textcolor{black}{$q_3$}};

			\node[] at (-0.5,0.2){\textcolor{gray}{${q_0}$}};
			\node[] at (-1.5,0.2){\textcolor{gray}{${q_{-1}}$}};
			\node[] at (-2.5,0.2){\textcolor{gray}{${q_{-2}}$}};
			\node[] at (-3.5,0.2){\textcolor{gray}{${q_{-3}}$}};
			
			\node[] (q0) at (-0.5,-0.3){\textcolor{black}{$\nicefrac{q_3}{\mu}$}};
			\node[] (q-1) at (-1.5,-0.3){\textcolor{black}{$\nicefrac{q_2}{\mu}$}};
			\node[] (q-2) at (-2.5,-0.3){\textcolor{black}{$\nicefrac{q_1}{\mu}$}};
			\node[] (q-3) at (-3.5,-0.3){$\nicefrac{\textcolor{gray}{q_0}}{\mu}$}; 
			\node[] (q-5) at (-3.4,-0.7){$\nicefrac{q_3}{\mu^2}$};

			\draw[-] (0,0.2) --  (0,-0.2) coordinate[label = {below: $0$}] (T);
			
			\draw[decorate,decoration={brace,amplitude=5pt},gray] (-4,0.5) -- (0,0.5)  node[black,midway,yshift=+0.1cm] (q) {}; 
			
			\draw[decorate,decoration={brace,amplitude=5pt},gray] (-1,0.75) -- (3,0.75)  node[black,midway,yshift=+0.1cm] (uq) {};
			\draw[<->,gray] (q)  to[bend left]  (uq) ;
			\node[] at (-0.5,1.5){$\textcolor{gray}{\phi}=\nicefrac{\mathscr{U}\phi}{\mu}$};

			\node[circle, fill, minimum size=2,inner sep=1.2pt, outer sep=1pt, gray] at (1,0){};
			\node[circle, fill, minimum size=2,inner sep=1.2pt, outer sep=1pt, gray] at (2,0){};
			
			\draw[-,gray] (0.1,0.1) --  (-0.1,-0.1);
			\draw[-,gray] (0.1,-0.1) --  (-0.1,0.1);
			\draw[-,gray] (3.1,0.1) --  (2.9,-0.1);
			\draw[-,gray] (3.1,-0.1) --  (2.9,0.1);    
			
			\draw[<->,thin,gray] (0.15,-0.15)  to[bend right]  (2.85,-0.15);
			\node[] at (1.4,-0.73){{\footnotesize\textcolor{gray}{$q(0)=\mu q(T)$}}};
			
			\end{tikzpicture}}
	\end{subfigure}
	\caption{ Clarification to the proof  of Lemma~\ref{th:operator} for two examples. Function $\phi$ represents an eigenfunction of the monodromy operator. Functions $q_1,\ldots,q_N$ in \eqref{eq:Asmu} can be interpreted as describing segments of the emanating solution.} 
	\label{fig:th}
\end{figure}

\smallskip

Denoting $\bm{v}=q(0)$, conditions \eqref{eq:difeq} can be rewritten in the form
\begin{equation}\label{eq:NLeig}
\No(\mu)  \bm{v}=0.
\end{equation}
Here, function  $\No:\ \C\rightarrow \C^{Nd\times Nd}$ is analytic in $\C\setminus\{0\}$, and for a given value of $\mu$ and $\bm{v}$ the matrix-vector product $\No(\mu)\bm{v}$ is determined as follows,
\begin{equation}\label{nleig2}
\No(\mu)\bm{v}=\bm{q}(1)-B(\mu)\bm{v},
\end{equation}
where $\bm{q}$ is the solution of initial value problem 
\begin{equation}
\begin{cases}
\dot{\bm{q}}(s)=A(s,\mu)\bm{q}(s), \qquad &s\in[0,\,1],\\
\bm{q}(0)=\bm{v}.
\label{eq:nldifeq}
\end{cases}
\end{equation}

Equation \eqref{eq:NLeig} can be interpreted as a nonlinear eigenvalue problem. The relation with the eigenvalue problem for the monodromy operator is clarified in the following theorem, which generalizes Theorem~2.1 of \cite{Skubachevskii2006} to non-scalar systems with multiple delays.

\begin{theorem}\label{th:eigenpair}
	Let $\hat \mu\in \C\setminus\{0\}$.  If the pair $(\hat \mu,\hat\phi)$ is a solution of the infinite-dimensional linear eigenvalue problem
	\begin{equation}\label{wm:infeig}
	\mathscr{U}\phi=\mu\;\phi,\ \mu\in\C,\ \phi\in X\setminus\{0\},
	\end{equation}
	then $(\hat\mu,\hat{\bm{v}})$ is a solution of the finite-dimensional nonlinear eigenvalue problem
	\begin{equation}\label{wm:nleig}
	\mathcal{N}(\mu)\bm{v}=0,\ \  \mu\in\C,\ \bm{v}\in\C^{Nd},
	\end{equation}	
	where $\hat{\bm{v}}=(v_1^\transpose\,\cdots\,v_N^\transpose)^\transpose$ is determined by 
	\begin{equation}
	v_n=x((n-1)\Delta;0,\hat\phi),\qquad n=1,\ldots,N.\label{eq:vnsol}
	\end{equation}
	
	Conversely,  if the pair $(\hat\mu,\hat{\bm{v}})$ is a solution of \eqref{wm:nleig}, then $(\hat\mu,\hat\phi)$ is a solution of \eqref{wm:infeig},  where 
	\[
	\hat\phi(t) =  {\hat{\mu}}^{\left\lfloor\frac{n-1}{N} \right\rfloor} q_{(n-1)\bmod_N+1}\left(\frac{t}{\Delta}-(n-1)  \right),\ t\in ((n-1)\Delta,\, n\Delta],\, n=-n_h+1,\ldots,0,
	\]
	with $\bm{q}(s)=\big(q_1^\transpose(s)\,\cdots\,q_N^\transpose(s)\big)^\transpose$  the solution of initial value problem \eqref{eq:nldifeq} for $\bm{v}=\hat{\bm{v}}$ and $\mu=\hat{\mu}$.
\end{theorem}
\begin{proof}  The assertions follows from the constructions in the proof of Lemma~\ref{th:operator}.
\end{proof}

The dimension of nonlinear eigenvalue problem \eqref{wm:nleig} does not only depend on the system dimension $d$ but also on the number $N$ of $\Delta$-subintervals of $[0,\,T]$. The latter is minimized if $\Delta$, in Assumption~\ref{as1}, is chosen as the greatest common divisor of $\tau_1$, $\ldots$, $\tau_h$, and $T$. We now present an example where $N=1$.

\begin{example}\label{ex:nleig} We consider system
	\begin{equation}
	\dot{x}(t)=\sum\nolimits_{j=0}^hA_j(t)x(t-j\tau)\label{eq:commensurate}
	\end{equation}
	with $T=\tau$.
	When taking $\Delta=\tau$ we can express $\No(\mu)v=q(1)-\mu\; v$, where $q$ satisfies
	\begin{equation}
	\begin{cases}
	\dot{q}(s)=\tau \sum_{j=0}^h A_j(s\tau)\frac{q(t)}{\mu^j}, \quad s\in[0,\, 1],\\
	q(0)=v.
	\end{cases}
	\label{eq:diffcomm}
	\end{equation}
	
	Moreover, if system \eqref{eq:commensurate} is scalar, \ie $d=1$, $\No(\mu)$ can be explicitly expressed in terms of the coefficients $A_j(t)$ and we can derive a nonlinear equation for the Floquet multipliers. Indeed, if we let 
	\[
	a_j=\int_0^1 \tau A_j(s\tau)\dd s=\int_0^\tau A_j(t)\dd t,
	\]
	then the solution of the differential equations in \eqref{eq:diffcomm} is 
	\[	q(1)= \exp\left(\sum\nolimits_{j=0}^h\mu^{-j}a_j\right)v.\]
	Hence, the Floquet multipliers correspond to the solutions of equation 
	\begin{equation}
	\exp\left(\sum\nolimits_{j=0}^h\mu^{-j}a_j\right)-\mu=0,\label{eq:d1N1}
	\end{equation}		
	which is in accordance  to the characteristic equation for the Floquet exponents as defined in \cite[Section~8.1]{Hale1977}.
\end{example}

\section{Floquet multipliers computation and stability assessment}\label{sec:Stabass}
This section describes two techniques for computing dominant  Floquet multipliers and stability assessment, outlined in Figure~\ref{fig:scheme}. The first one, discussed in section~\ref{sec:discr1st}, approximates the Floquet multipliers by a spline collocation method, furnishing a global view of  the monodromy operator spectrum, and, hence, detecting a guess of the largest Floquet multiplier in modulus. The second technique, analyzed in section~\ref{sec:approxnleig}, computes the Floquet multiplier in a neighborhood of an initial guess by Broyden's method, a local root-finding method, where the accuracy of the computation depends on evaluating the matrix-vector product $\No(\mu)\bm{v}$. Finally, section~\ref{sec:FMcomp} discusses an algorithm for the computation of Floquet multipliers based on a combination of these techniques.
\begin{figure}[h]	
	
	\tikzstyle{block} = [rectangle, text width=4cm, text badly centered]
	
	\begin{tikzpicture}[>=stealth'] 
	
	\node [block,text width=4cm] (infeq) {{\footnotesize{Infinite-dimensional linear eigenproblem}}\\ $\mathscr{U}\varphi=\mu\varphi$, \quad  $\varphi\in X$};
	
	\node [block,text width=4cm, right=4cm of infeq] (fineq) {{\footnotesize{Finite-dimensional non-linear eigenproblem}}\\ $\mathcal{N}(\mu)\bm{v}=0$, \quad $\bm{v}\in \mathbb{C}^{Nd}$};

	\draw [<->,thick] (infeq)-- (fineq) node[midway,above] (c) {{\footnotesize{Theorem~\ref{th:eigenpair}}}};

	
	\node [block,text width=4cm,  below=1.2cm of infeq] (infeig) {{\footnotesize{Matrix eigenvalue problem~\eqref{eq:Minfeig} \\ $\rightarrow$ \emph{global approximation}    }}};

	\draw [->] (infeq)-- (infeig) node[midway,sloped,above] (c) {}; 
	
	\draw [dashed,->] (fineq)-- (infeig) node[midway,sloped,above] (c) {}; 
	
	\node [block,text width=4cm, below=1.2cm of fineq] (finsim) {{\footnotesize{Broyden's method for solving $\mathcal{N}_\delta(\mu)\bm{v}=0$ \\ $\rightarrow$ \emph{local corrrection} }}};
	\draw [->] (fineq)-- (finsim) node[midway,sloped,above] {}; 
	\end{tikzpicture}
	\caption{Approach for computing dominant Floquet multipliers and stability assessment of~\eqref{eq:system}. The upper part illustrates the theoretical results of section~\ref{sec:diffeq} concerning the dual interpretation of  Floquet multipliers in terms of eigenvalue problems. The lower part outlines the Floquet multipliers computation of section~\ref{sec:Stabass}.} 
	\label{fig:scheme}
\end{figure}
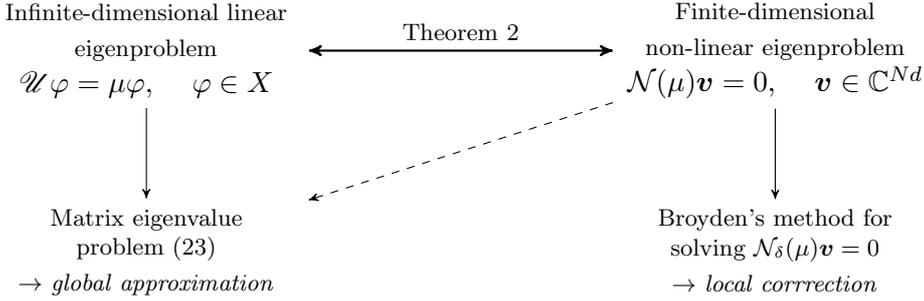

\subsection{Discretization the infinite-dimensional linear eigenvalue problem}\label{sec:discr1st}
We derive a finite-dimensional linear eigenvalue problem, whose eigenvalues approximate Floquet multipliers.  First we outline how approximate solutions of \eqref{eq:system} in the form of a spline can be computed by collocation approach. 
 Second, we show how these approximations induce a matrix approximation of the monodromy operator.
 The derivation is different from~\cite{Breda2015} and from \cite{Butcher2011} for the one-delay case, but the underlying ideas are the same.    Third, in the spirit of Theorem~\ref{th:eigenpair}, we show that the corresponding linear eigenvalue problem can also be obtained from a particular approximation of the nonlinear eigenvalue problem \eqref{wm:nleig}. The latter explains a particular choice of spline, in accordance with the subdivision of the interval $[-\tau_h,\, T]$ in intervals of length $\Delta$.  We conclude with some implementation aspects.

\paragraph*{Discretization of the initial value problem} 
We approximate a solution of \eqref{eq:system}
by a spline, which is piecewise defined on $\Delta$-subintervals by $M$-degree polynomials 
\begin{equation}\label{wm:pol}
\{q_n^M\}_{n=-n_h+1}^N,
\end{equation}
where polynomial $q_n^M$ represents the approximation on the interval $[(n-1)\Delta,\, n\Delta)$. Without loss of generality, we assume that this $\Delta$-interval is scaled and shifted to the interval $[0,\,1)$. 
Given a polynomial basis $\{p_i\}_{i\in\N}$ in the interval $[0,\,1]$, this spline approximation  is uniquely determined by the coefficients $c$ of the polynomials \eqref{wm:pol}, 
\begin{equation*}
c=\begin{pmatrix}
c_{-n_h+1}\\\vdots\\ c_N
\end{pmatrix},\ \quad c_n=\begin{pmatrix}
c_{0,n}\\\vdots\\ c_{M,n}
\end{pmatrix}, \ n=-n_h+1,\ldots, N,
\end{equation*}
such that
\begin{equation}\label{wm:parsol}
q_n^M(s)=\sum_{i=0}^Mc_{i,n}p_i(s), \quad s\in[0,\,1]. 
\end{equation}
We note that the coefficient vector  $(c_{-n_h+1}^\transpose\, \cdots\, c_0^\transpose )^\transpose$ 
can be interpreted as a parametrization of the initial condition. Specifying this vector to the value $c_{\phi}\in\R^{n_h d(M+1)}$ leads us to the condition
\begin{equation}\label{wm:Bc}
Bc=c_\phi,
\end{equation}
where 
\begin{equation}\label{wm:defB} 
B=\begin{pmatrix}I_{n_h} & 0_{n_h\times N} \end{pmatrix}\otimes I_{d(M+1)}.
\end{equation}
We now determine conditions expressing that vector $(c_{1}^\transpose\, \cdots\, c_N^\transpose )^\transpose$ corresponds to the emanating solution. We first require continuity on $[0,\, T]$, which can be expressed as
\begin{equation}\label{wm:cont}
q^M_n(0)=q^M_{n-1}(1),\quad n=1,\ldots, N.
\end{equation}
Second, we impose collocation requirements for differential equation \eqref{eq:system} on a mesh $\{ \xi_m\}_{m=1}^M$ over the interval $[0,\, 1]$, which leads to
\begin{equation}\label{wm:col}
\dot{q}^M_n(\xi_m)=\Delta\sum_{j=0}^hA_j\left((\xi_m+n-1)\Delta\right)q^M_{n-n_j}(\xi_m),\ \ m=1,\ldots, M,\; n=1,\ldots, N.
\end{equation}

Conditions \eqref{wm:cont} and \eqref{wm:col} can be be stated in matrix form as
\begin{equation}\label{wm:Sc}
Sc=0,
\end{equation}
with $S\in\C^{(N\times (n_h+N))d(M+1)}$. 
Finally, the approximate solution with initial condition parametrized by $c_\phi$ can be computed from \eqref{wm:Sc} and \eqref{wm:Bc}, leading to
\begin{equation}\label{wm:spline}
c=\begin{pmatrix}S \\ B \end{pmatrix}^{-1}E^\transpose c_\phi,  
\end{equation}
with
\[
E=\begin{pmatrix}0_{n_h\times N} & I_{n_h} \end{pmatrix}\otimes I_{d(M+1)}.
\]

\paragraph*{Discretization of the infinite-dimensional linear eigenvalue problem} 

The monodromy operator describes the translation along a solution from the interval $[-\tau_h,\, 0]$ to the interval $[T-\tau_h,\, T]$. Considering the previously defined spline approximation of a solution, the action of the monodromy operator can be approximated by the mapping from vector $c_\phi$, which gives rise to discretized solution \eqref{wm:spline}, into vector $c_T=\left(c_{N-n_h+1}^\transpose  \, \cdots \, c_N^\transpose\right)^\transpose$. 

We can express $c_T=E c$, leading to
\[
c_T=E \begin{pmatrix}S \\ B \end{pmatrix}^{-1}E^\transpose\ c_\phi. 
\]
Hence, the discretization of the infinite-dimensional linear eigenvalue problem \eqref{eq:infeig} is 
\begin{equation}\label{eq:Minfeig}
\mathscr{U}_M c_\phi=\mu\; c_\phi,
\end{equation}
where 
\[
\mathscr{U}_M
=
E \begin{pmatrix}S \\ B \end{pmatrix}^{-1}E^\transpose= \begin{pmatrix} 0 &   I_{n_h d(M+1)}  \end{pmatrix}
\begin{pmatrix}S \\  I_{n_h d(M+1)}\ \ 0  \end{pmatrix}^{-1}
\begin{pmatrix} 0\\I_{n_h d(M+1)} \end{pmatrix}
\]
is a matrix approximation of the monodromy operator.

\paragraph*{Interpretation in terms of the finite-dimensional nonlinear eigenvalue problem} 

We can alternative obtain \eqref{eq:Minfeig} from the nonlinear eigenvalue problem \eqref{wm:nleig} by a spectral discretization of boundary value problem \eqref{eq:difeq}. More precisely, approximating $q$ by a polynomial of degree $M$,  
\[
q^M(s)=\begin{pmatrix} 
q_1^M(s)\\
\vdots \\
q_N^M(s)
\end{pmatrix},\ \ s\in [0,\, 1],
\]
imposing collocation requirements for the differential equation on the mesh $\{\xi_m\}_{m=1}^M$ and imposing the condition $q^M(1)=B(\mu) q^M(0)$ leads us to \eqref{wm:Sc}
supplemented with
\begin{equation}\label{wm:extra2}
c_\phi= \mathcal{R}(\mu) c_q,
\end{equation}
where
\[
c_\phi=\begin{pmatrix}c_{-n_h+1} \\  \vdots \\ c_0\end{pmatrix}, \quad  c_q=\begin{pmatrix}c_1\\ \vdots\\ c_N\end{pmatrix},
\]
and
\[
\mathcal{R}(\mu)=\begin{cases}
\begin{pmatrix} 0_{N-n_h\times n_h} &  I_{n_h} \end{pmatrix}\otimes I_{d(M+1)},  & n_h\leq N, \\
\begin{pmatrix}
0\ \  \frac{I_{n_h\bmod_N}}{\mu^{\left\lceil\nicefrac{n_h}{N}\right\rceil}}\\
{\frac{I_{N}}{\mu^{\left\lceil\nicefrac{n_h}{N}\right\rceil-1}}}  \\
{\vdots} \\
{\nicefrac{I_{N}}{\mu}} 
\end{pmatrix}         \otimes I_{d(M+1)}, & n_h>N.
\end{cases}
\]
As spelled out in the proof of Lemma~\ref{th:operator},  differential equation \eqref{eq:difeq} is namely derived from  equation \eqref{eq:nshift} by expressing $x$ in the interval $[-\tau_h,\, 0]$ in terms of $\mu,\, \mu^2,\,\ldots$-fractions of $x$ at positive time-instants, 
where $x$ is a solution emerging from an eigenfunction of $\mathscr{U}$. The two cases in the expression for $\mathcal{R}(\mu)$ correspond  to the situation  depicted in the left, respectively  right pane of Figure~\ref{fig:th}.

Partitioning $S=\left(S_{\phi}\ S_q\right)$, according to the subdivision of $c$ into $c_\phi$ and $c_q$,  allows to rewrite \eqref{wm:Sc}-\eqref{wm:extra2} as
\[
\begin{cases}
S_\phi c_\phi+S_q c_q=0,\\
c_\phi= \mathcal{R}(\mu)  c_q,
\end{cases}
\]
which brings us to the polynomial eigenvalue problem
\begin{equation}\label{eq:Mfineig}
\mu^{\left\lceil\nicefrac{n_h}{N}\right\rceil}\left(S_\phi \mathcal{R}(\mu)+S_q\right)c_q=0.
\end{equation}
To establish a connection between this eigenvalue problem and \eqref{eq:Minfeig}, we note that \eqref{wm:extra2} is equivalent to
\[
B\; c=\frac{1}{\mu} E\; c.
\] 
Conditions  \eqref{wm:Sc}-\eqref{wm:extra2} can then be
rewritten  as
$
\mu S c=0$ and $\mu B c=E c$, which brings us to the generalized eigenvalue problem
\begin{equation}\label{wm:pencil}
\left(\mu \begin{pmatrix}S \\ I_{n_h d(M+1)}\ 0 \end{pmatrix} -\begin{pmatrix} 0_{Nd(M+1)} & 0\\ 0 & I_{n_hd(M+1)} \end{pmatrix}\right)\; c=0.
\end{equation}
From Sylvester's determinant identity it follows that the sets of nonzero eigenvalues of~\eqref{wm:pencil} and of monodromy matrix approximation  $\mathscr{U}_M$ are the same.

\smallskip
In conclusion, the nonzero eigenvalues of  matrix $\mathscr{U}_M$, generalized eigenvalue problem \eqref{wm:pencil},  and polynomial eigenvalue problem \eqref{eq:Mfineig} coincide. The dimensions are  $n_hd(M+1)$, $(n_h+N)d(M+1)$, and $N d(M+1)$ (but the order is $\left\lceil\nicefrac{n_h}{N}\right\rceil$) while the corresponding eigenvectors, $c_\phi,\ c,\ c_q$ parametrize segments of the approximate solution emerging from the eigenfunction.

\paragraph*{Choice of polynomial basis and collocation points}
Generalized eigenvalue problem \eqref{wm:pencil} and polynomial eigenvalue problem \eqref{eq:Mfineig} lie at the basis of our algorithm for approximating Floquet multipliers. 
In our implementation, we express polynomials $\{q_n^M\}_{n=-n_h+1}^N$ in a Chebyshev basis,
\begin{equation*}
q^M_n(t)=\sum_{i=0}^M c_{i,n}  T_i\left(2 t-1\right), \quad t\in[0,\,1],
\end{equation*}
where $T_i$ is the $i$-degree Chebyshev polynomial of the first kind, and we take as collocation points $\{\xi_m\}_{m=1}^M$ the corresponding  Chebyshev nodes,
\begin{equation}\label{wm:mesh}
\xi_m=\frac{1}{2}(\alpha_{m}+1),\ \   \alpha_{m}=-\cos\frac{(m-1)\pi}{M},\ m=1,\ldots,M.
\end{equation}
As we shall document in section~\ref{sec:example1} spectral accuracy (approximation error $O(M^{-M})$) is observed for the individual Floquet multiplier approximations. This is expected, as the method can be interpreted in terms of a spectral discretization~\cite{Trefethen2000} of boundary value problem \eqref{eq:difeq} with \eqref{wm:mesh} as collocation points.

To conclude the section we specify matrix $S$ in \eqref{wm:pencil} for a particular case where $N=1$.
\begin{example}
	We reconsider System~\eqref{eq:commensurate} with $T=\tau$, leading to $\Delta=1$. We can express 
		\[S=
			\begin{pmatrix} 
				0 & \cdots & 0 & -(1\ 1\ \cdots\ 1)\otimes I_d & 	(1\ (-1) \ \cdots\ (-1)^{M}) \otimes I_d\\
		-\bm{A_{h}}& \cdots &	-\bm{A_{2}} & 	-\bm{A_{1}} & \bm{U}\otimes I_d-\bm{A_0}
	\end{pmatrix},
	\]
	where 
	\[
	\bm{A_j}=\Delta\begin{pmatrix}
	A_j(\xi_1\Delta)T_0\left(\alpha_1\right)&\ldots &A_j(\xi_1\Delta)T_M\left(\alpha_M\right)\\
	\vdots& 	&\vdots\\
	A_j(\xi_M \Delta)T_0\left(\alpha_M\right)&\ldots &A_j(\xi_M\Delta)T_M\left(\alpha_M\right)
	\end{pmatrix},
	\]
	for $j=0,\ldots, M$, and
	\[		%
	\bm{U}=2\begin{pmatrix}
	0& {1}U_0\left(\alpha_1\right)&\ldots &      M U_{M-1}\left(\alpha_1\right)\\
	\vdots& \vdots& 	&\vdots\\
	0&{1}U_0\left(\alpha_M\right)&\ldots & M U_{M-1}\left(\alpha_M\right)
	\end{pmatrix},
	\]
	with  $U_i$ the $i$-degree Chebyshev polynomial of the second kind.
\end{example}

\subsection{Newton type algorithms for the nonlinear eigenvalue problem}\label{sec:approxnleig}

The alternative formulation \eqref{wm:nleig} of the eigenvalue problem allows us to compute eigenpairs by applying an iterative solver for nonlinear equations to the system
\[
\begin{cases}
\No(\mu)\bm{v}=0 \\
\bm{w}^*\bm{v}=1,
\end{cases} 
\]
where the second equation, with $\bm{w}\in\C^{Nd},$ is a normalization constraint.  Letting $x=(\bm{v}^\transpose\, \mu)^\transpose$ we can compactly write the system in the form $F(x)=0$.  The application of a damped Newton method leads us to the basic iteration 
\begin{equation*}
\begin{cases}
y_i=-H_i F(x_i) & \\
x_{i+1}=x_i+\gamma_i y_i, & \ i=0,1,2,\ldots
\end{cases}
\end{equation*}
where $\gamma_i\in (0,1)$ is the damping factor and $H_i$ represents the employed inverse Jacobian approximation.

In the exact Newton method the true Jacobian is used, which implies
\[
H_i=J(x_i)^{-1}= \begin{pmatrix} \No(\mu_i)& \dv{\No(\mu_i)}{\mu}\bm{v_i}\\ 
\bm{w}^* &0
\end{pmatrix}^{-1}
\]
and a system of equations needs to be solved to obtain direction $y_i$.
To compute $\No(\mu_i)$ we have to solve initial value problem \eqref{eq:nldifeq} from $Nd$ independent initial conditions. To compute matrix-vector product  $\dv{\No(\mu_i)}{\mu}\bm{v_i}$ we can use the following proposition.
\begin{proposition}\label{lemma:deriv} 
	Let $(\mu,\bm{v})$, with $\mu\neq 0$, be an eigenpair of \eqref{wm:nleig}. 
	Let $\bm{q}$ be such that \eqref{eq:nldifeq} is satisfied.  Then we can express
	\[
	\dv{\No(\mu)}{\mu}\bm{v}=\bm{q}_\mu(1)-\bm{v_N},
	\]
	where  $\bm{q}_\mu$ is the solution of initial value problem
	\begin{equation}
	\begin{cases}
	\dot{\bm{q}}_\mu(s)=\pdv{A(s,\mu)}{\mu}\bm{q}(s)+A(s,\mu)\bm{q}_\mu(s),\quad s\in[0,\, 1],\\
	\bm{q}_\mu(0)=\bm{0}
	\end{cases}
	\label{eq:diffmu}
	\end{equation}
	and  $\bm{v_N}=(0\, \cdots\, 0\,v_1^\transpose)^\transpose$, with   $ \bm{v}=(v_1^\transpose\, \cdots\, v_N^\transpose)^\transpose$.
\end{proposition}
\begin{proof}
	The solution of \eqref{eq:nldifeq}, evaluated at any $s\in[0,\, 1] $, is a smooth function of $\mu$ on $\C\setminus\{0\}$. The assertion follows from differentiating \eqref{nleig2} and \eqref{eq:nldifeq} with respect to $\mu$, the latter giving rise to  variational equation \eqref{eq:diffmu}.
\end{proof}

In the Broyden's quasi-Newton method \cite{Al-Baali2013,Jarlebring2018}, more precisely the so-called \emph{good Broyden's method}, an approximation of the Jacobian inverse is updated in every iteration by a rank one matrix, using the formula
\begin{equation*}
\quad H_{i+1}=H_i-\frac{(H_iF(x_{i+1})+(1-\gamma_i) y_i)\; (y_i^* H_i)}{y_i^*(H_iF(x_{i+1})+y_i)},
\end{equation*}  
where in our implementation we use as initialization  $H_0= I$ or 
\begin{equation}\label{wm:contH}
H_0= J(x_0)^{-1}.
\end{equation}

Since characteristic matrix $\No(\mu)$ is not in explicit form,  we need to numerically compute \eqref{eq:nldifeq} and \eqref{eq:diffmu} for  evaluating matrix-vector products $\No(\mu) v$ and $\frac{d\No(\mu)}{d\mu} v$.
For stiff and non-stiff systems, our implementation  is  based on the use of a Runge-Kutta method of order $4$, and trapezoidal rule  with fixed step-size~$\delta$, respectively. In this way the resulting iterative scheme can also be interpreted as applying a Newton type method to solve the approximate eigenvalue problem
\begin{equation}\label{wm:Ndelta}
\No_\delta(\mu)\bm{v}=0,
\end{equation}  
where $\No_\delta$ is obtained from $\No$ by replacing the exact solution of \eqref{eq:nldifeq} and \eqref{eq:diffmu} by the numerical solution.  
Compared to standard nonlinear eigenvalue problems in explicit form, there is hence a significant additional cost in evaluating the characteristic matrix $\No(\mu)$, which is needed in every iteration of Newton method but avoided in Broyden's method. This  explains why sometimes explicitly computing the inverse in initialization \eqref{wm:contH} may be beneficial.

\subsection{Floquet multipliers computation} \label{sec:FMcomp}	
The results in sections~\ref{sec:discr1st}-\ref{sec:approxnleig} lead us to the following high-level algorithm for computing (part of) the spectrum of the monodromy matrix,
which is conceptually similar to the approach of \cite{Michiels2011} for computing characteristic roots of linear time-invariant systems with delays.
\begin{algo2}\textit{Two-stage approach for computing Floquet multipliers.}\label{algo:FMsmall}
	\begin{enumerate}
		\item Fix $M$ and compute eigenvalues and eigenvectors of \eqref{eq:Mfineig}.
		\item Fix $\delta $ and correct the individual Floquet multiplier approximations, and extracted eigenvector approximations of \eqref{wm:nleig}, by applying Broyden's method to \eqref{wm:Ndelta}.
	\end{enumerate}
\end{algo2}

For problems of moderate dimension~$N d$, we use a direct method to compute all eigenvalues of \eqref{eq:Mfineig} in the first stage, resulting in approximate eigenpairs  $(\mu,c_q)$.   Recall from \eqref{wm:parsol} and the construction in the previous section  that $c_{q}=(c_{1}^\transpose\,\cdots\, c_N^\transpose)^\transpose$ parametrizes the approximate solution corresponding to an eigenfunction of $\mathscr{U}$.  By Lemma~\ref{th:operator} this allows us to extract an approximation of the eigenvector $\bm{v}$ in \eqref{wm:nleig}, 
\[
\bm{v}=(v_1^\transpose\, \cdots \, v_N^\transpose)^\transpose,\ \  v_n=\sum_{i=0}^Mp_i(0)c_{i,n}, \ n=1,\ldots,N,  
\]
that is used in the initialization of the second stage, along with \eqref{wm:contH} as initialization of the inverse Jacobian approximation.

For problems with high dimension~$Nd$, it is computationally infeasible to use a direct solver in the first stage of Algorithm~\ref{algo:FMsmall}, and iterative eigensolvers are to be preferred. In section~\ref{sec:example3}, we employ Arnoldi method to matrix $\mathscr{U}_M$  in \eqref{eq:Minfeig}, where every iteration requires solving a system of equation with matrix $S_q$ since
\begin{align}
\begin{split}\label{eq:Arnoldi}
\mathscr{U}_Mc_\phi&=\begin{pmatrix} 0&  I_{n_hd(M+1)} \end{pmatrix}\begin{pmatrix} S_\phi & S_q\\ I & 0\end{pmatrix}^{-1}\begin{pmatrix} 0\\  I_{n_hd(M+1)} \end{pmatrix}c_\phi
\\&=\begin{pmatrix} 0&  I_{n_hd(M+1)} \end{pmatrix}\begin{pmatrix} c_\phi \\ -S_q^{-1}S_\phi c_\phi\end{pmatrix}.
\end{split}
\end{align}
In this way eigenvector approximation are obtained in the form of coefficient vector  $c_\phi$.
The latter can be turned into an eigenvector approximation of \eqref{wm:nleig} by numerically solving initial value problem \eqref{eq:ivp} on the interval $[0,\,T]$, with $\phi$ the $M$-degree spline defined by the coefficients $c_\phi$. The inverse Jacobian in Broyden's method is initialized with the identity matrix.

\smallskip

Algorithm~\ref{algo:FMsmall} turns out to be very efficient for computing dominant Floquet multipliers.  As the first step serves to scan the complex plane, the basic requirement on $M$ is that the obtained approximation of the dominant Floquet multiplier and corresponding eigenvector is in the region of attraction of Broyden's method. As the second step serves for local improvement (and discarding spurious eigenvalues), parameter $\delta$ should be chosen sufficiently small not only to effectively take the role of corrector but to be able to reach the desired final accuracy on the Floquet multipliers (recall that modified problem \eqref{wm:Ndelta} is actually solved in the second stage).
\begin{remark} As we have seen in section~\ref{sec:discr1st}, the  discretized eigenvalue problem 
	\eqref{eq:Mfineig} can also be derived from \eqref{wm:nleig} by approximating the solution of the differential equation in \eqref{eq:nldifeq}. This approach was based on  solving \emph{boundary value problem} \eqref{eq:difeq}, parametrized by $\mu$, for which the collocation approach is appropriate. In the context of local corrections we need to solve \emph{initial value problem} \eqref{eq:nldifeq} for specified values of $\mu$ and $\bm{v}$, for which a time-stepping method is preferable. 
\end{remark}

\begin{remark}
	If a factorization of matrix $S_q$ or even storing the matrices of the discretized eigenvalue problem would be infeasible, 
	one may be able to apply Broyden's method based on the nonlinear eigenvalue problem formulation (the second stage). In this case, multiple Floquet multiplier approximations can still be obtained,  using the deflation technique as presented in \cite{Effenberger2013,Jarlebring2018}. 
\end{remark}

\section{Characterization and computation of left eigenvectors} \label{sec:left}

We again consider nonlinear eigenvalue problem \eqref{wm:nleig}, with $\mathcal{N}$ defined by \eqref{nleig2}-\eqref{eq:nldifeq}. 		
Vector $\bm{u}\in\C^{Nd}\setminus\{0\}$ is  called a left eigenvector of $\mathcal{N}$, corresponding to eigenvalue $\mu$ if 
$$
\bm{u}^* \mathcal{N}(\mu)=0,
$$ 
or, equivalently,  $\mathcal{N}(\mu)^* \bm{u}=0$.  In order to give an expression in terms of a right eigenvector, we first define the ``transposed" nonlinear eigenvalue problem
\begin{equation}\label{wm:dual2}
\Mo(\mu) \bm{v}=0,
\end{equation} 
where for given $\mu\in\C\setminus\{0\}$ and $\bm{v}\in\C^{Nd}$ the matrix vector product $\Mo(\mu) \bm{v}$  is determined as
\begin{equation}\label{wm:dual}
\Mo(\mu) \bm{v}=\bm{p}(0)-B(\mu)^\transpose \bm{v},
\end{equation}
where $\bm{p}$ is the solution of problem 
\begin{equation}
\begin{cases}
\dot{\bm{p}}(s)=-A(s,\mu)^\transpose\bm{p}(s), \quad s\in[0,\,1],\\
\bm{p}(1)=\bm{v}.
\end{cases}
\label{eq:adjoint}
\end{equation}
This bring us to the following result.
\begin{theorem} 		\label{th:adjoint}
	Let $\bm{u}\in\C^{Nd}$ and $\mu\in\C\setminus\{0\}$. Vector $\bm{u}$ is a left eigenvector of \eqref{wm:nleig} corresponding to  eigenvalue $\mu$ if and only if it is a right  eigenvector
	of  problem \eqref{wm:dual} corresponding to eigenvalue $\bar\mu$.  
\end{theorem} 
\begin{proof}
	It suffices to prove that for any $\mu\in\C\setminus\{0\}$ we have
	$
	\Mo(\mu)=\No(\bar\mu)^*. 
	$	
	For this,	
	take  arbitrary vectors $\bm{u} \in\C^{Nd}$ and $\bm{v}\in\C^{Nd}$.	With  $\langle \cdot, \cdot \rangle$  the Euclidean inner product in $\C^{Nd}$ and with $\bm{q}$ satisfying
	\[
	\begin{cases}
	\dot{\bm{q}}(s)=A(s,\bar\mu)\bm{q}(s), \qquad &s\in[0,\,1],\\
	\bm{q}(0)=\bm{u},
	\end{cases}
	\]
	we can derive
	\[
	\begin{array}{rcl}
	\langle \No(\bar \mu)\bm{u},\bm v\rangle -\langle \bm{u}, \Mo(\mu)\bm{v}\rangle 
	&=&\bm{q}(1)^* \bm{v}-\bm{u}^*B(\bar \mu)^* \bm{v} - \bm{u}^* \bm{p}(0)  +\bm{u}^*B(\mu)^\transpose \bm{v} \\
	&=&\bm{q}(1)^* \bm{v}-\bm{u}^* \bm{p}(0) \\
	&=&\bm{q}(1)^*\bm{p}(1)-\bm q(0)^* \bm{p}(0) \\
	&=&\int_0^1  \left(\bm{q}(s)^* \bm{p}(s)\right)^\prime\dd s\\
	&=&\int_{0}^{1} \dot{\bm{q}}(s)^* \bm{p}(s) +\bm{q}(s)^* \dot{\bm{p}}(s)\dd s \\
	&=&\int_{0}^{1}   \bm{q}(s)^* A(s,\bar \mu)^*\bm{p}(s)  - \bm{q}(s)^* A(s,\mu)^\transpose \bm{p}(s)  \dd s\\
	&=&0,
	\end{array}
	\]
	which concludes the proof.
\end{proof}
\begin{example} For time-delay system \eqref{eq:commensurate} with $\tau=T$ we have
	$\Mo(\mu) v=p(0)-\mu v$ where
	\begin{equation*}
	\begin{cases}
	\dot{p}(t)=-\tau\sum_{j=0}^h A_j(s\tau)^\transpose\;\frac{p(t)}{\mu^j}, \quad s\in[0,\,1];\\
	p(1)=v.
	\end{cases}
	\end{equation*}
\end{example}

In the spirit of Theorem~\ref{th:eigenpair} we now establish, as the main result of this section, a connection between transposed nonlinear eigenvalue problem  \eqref{wm:dual} and the monodromy operator corresponding to the following periodic time-delay system:
\begin{equation}\label{wm:system:dual}
\dot x(t)=\sum_{j=0}^h A_j^\transpose(-t+\tau_j)\; x(t-\tau_j),
\end{equation}
where, in addition to taking the transpose of the coefficient matrices, a reserve and shift of time have taken place in these coefficient matrices.
In the formulation we will make use of the permutation matrix  $R\in\R^{Nd\times Nd}$ defined as
\[
R=\begin{pmatrix} && I_d \\ &\iddots &  \\ I_d &&  \end{pmatrix}.
\]
\begin{theorem}\label{th:dual}
	Let $\mathscr{D}$ be the monodromy matrix corresponding to system 
	\eqref{wm:system:dual}. Let $\hat \mu\in \C\setminus\{0\}$. 
	If the pair $(\hat \mu,\hat \phi)$ is a solution of 
	\begin{equation}\label{wm:infeig2}
	\mathscr{D}\phi=\mu\;\phi,\ \mu\in\C,\ \phi\in X\setminus\{0\},
	\end{equation}
	then $(\hat \mu,\hat{\bm{v}})$ is a solution of 
	\begin{equation}\label{wm:nleig2}
	\Mo(\mu)\bm{v}=0,\ \  \mu\in\C,\ \bm{v}\in\C^{Nd},
	\end{equation}	
	where   $\hat{\bm{v}}=(v_1^\transpose\,\cdots\,v_N^\transpose)^\transpose$ is determined by 
	\begin{equation}
	v_n=x((N-n)\Delta;0,\hat \phi),\qquad n=1,\ldots,N,
	\end{equation}
	and $x(t;t_0,\phi)$ the solution of \eqref{wm:system:dual} with initial condition $\phi$ at time $t_0$.
	
	Conversely,  if the pair $(\hat\mu,\hat{\bm{v}})$ is a solution of \eqref{wm:nleig2}, then $(\hat\mu,\hat \phi)$ is a solution of \eqref{wm:infeig2},  where 
	\[
	\hat\phi(t) =  {\hat\mu}^{\left\lfloor\frac{n-1}{N} \right\rfloor} \hat{p}_{(n-1)\bmod_N+1}\left(\frac{t}{\Delta}-(n-1)\right),\ t\in((n-1)\Delta,\, n\Delta],\, n=-n_h+1,\ldots,0,
	\]
	and $\hat{\bm{p}}(s)=(\hat{p}_1^\transpose(s)\,\cdots\,\hat{p}_N^\transpose(s))^\transpose$ satisfies 
	\begin{equation}\label{wm:defphat}
	\hat{\bm{p}}(s)=R \bm{p}(1-s),
	\end{equation} 
	with $\bm{p}$ the solution of initial value problem \eqref{eq:adjoint} for  $\bm{v}=R{\hat{ \bm{v}}}$ and $\mu={\hat{ \mu}}$.
\end{theorem}
\begin{proof}
	Let $\tilde{\Mo}(\mu)=R \Mo(\mu) R$ and $\tilde{ \bm{v}}=R\bm{v}$. In these variables,  using substitution \eqref{wm:defphat} and taking into account $RR=I$, nonlinear eigenvalue problem \eqref{wm:dual2}-\eqref{eq:adjoint} can be rewritten as  
	\begin{equation}\label{wm:nleigproof}
	\tilde{\Mo}(\mu) \tilde{\bm{v}}=0,
	\end{equation}
	where\[
	{\tilde{\Mo}}(\mu) \tilde{\bm{v}}=\tilde{\bm{p}}(1)-\left(RB(\mu)^\transpose R\right)\; \tilde{\bm{v}}
	\]
	and
	\[
	\begin{cases}
	\dot{\tilde{\bm{ p}}}(s)=\left(R A(1-s,\mu)^\transpose R\right)\; \tilde{\bm{ p}}(s), \quad s\in[0,\,1],\\
	\tilde{\bm{ p}}(0)=\tilde{\bm{ v}}.
	\end{cases}
	\]
	We have $$R B(\mu)^\transpose R=B(\mu),$$
	while the $d\times d$ block of $RA(1-s,\mu)^\transpose R$ at position $(dn,dk)$ for $n,k=0,\ldots,N-1$ satisfies
	\[
	\begin{array}{l}
	\Delta \sum_{\substack{j=0,\ldots,h\\ b_{n-n_j}=k} } 	A_j^\transpose\left( (N-k)\Delta +(1-s)\Delta\right)  \mu^{a_{n-n_j}} \\
	=
	\Delta \sum_{\substack{j=0,\ldots,h\\ b_{n-n_j}=k} }	A_j^\transpose\left( (N+(1-s)-b_{n-n_j}) \Delta \right)  \mu^{a_{n-n_j}} \\
	=
	\Delta \sum_{\substack{j=0,\ldots,h\\ b_{n-n_j}=k} }	A_j^\transpose\left( -(s+n-1)\Delta+n_j\Delta \right) \mu^{a_{n-n_j}}.
	\end{array}
	\] 
	The proof is complete by noting, from a comparison of this expression with \eqref{eq:Asmu}, that   nonlinear eigenvalue problem \eqref{wm:nleigproof} is equivalent to eigenvalue problem \eqref{wm:infeig2} in the sense of Theorem~\ref{th:eigenpair}. 
\end{proof}

Based on Theorems~\ref{th:adjoint}-\ref{th:dual} the methods for computing right eigenpairs as presented in the previous section can be trivially adapted to compute left eigenpairs.

\section{Floquet multipliers derivatives and stability optimization}\label{sec:deFloMustab}

In this section we assume that the matrices in \eqref{eq:system} depend on system or controller parameters $
\bK=(K_1\, \cdots \, K_k)^\transpose\in\R^k$. More precisely, we consider system
\begin{equation}
\dot{x}(t)=\sum_{j=0}^hA_j(t;\bK)x(t-\tau_j),
\label{eq:systemcontrol}
\end{equation}
under assumption that functions $A_j:\ \R\times \R^{k}\mapsto\R^{d\times d}$, $(t,\bK)\mapsto A_j(t;\bK)$ are smooth and $T$-periodic in the first argument, for $j=0,\ldots,h$. Note that we use a dot-comma to separate variables and parameters. Accordingly, differential equation \eqref{eq:difeq} changes to
\[
\dot{\bm{q}}(s)=A(s,\mu;\bK)\;\bm{q}(s)
\]
and we denote the nonlinear eigenvalue problem by $\No(\mu;\bK)\bm{v}=0$.
However, to simplify the notations, we will  omit the parametric argument whenever it is not essential for the understanding. 

Section~\ref{sec:deNLeig}  addresses the characterization of derivatives of simple Floquet multipliers with respect to parameters $\bK$, inferred from  nonlinear eigenvalue problem \eqref{wm:nleig}, and outlines their computation. 
These results are at the basis of a method for stability optimization of periodic systems with delay, presented in section~\ref{sec:stabopt}.

\subsection{Derivatives of Floquet multipliers}\label{sec:deNLeig}
The following proposition provides an explicit expression for derivatives of a Floquet multiplier with respect to parameters.
\begin{proposition}\label{th:pderiv}
	Let $(\mu,\bm{v})$, with $\mu\neq 0$, be an eigenpair of \eqref{wm:nleig}. 
	Assume that eigenvalue $\mu$ has multiplicity one and let $\bm{u}$ be the left eigenvector, \ie  $\bm{u}^*\No(\mu;\bK)=0$.
	Let $\bm{q}$ be such the \eqref{eq:nldifeq} is satisfied.  Then for each $j\in\{1,\ldots,k\}$ we can express
	\begin{equation}
	\pdv{\mu}{K_j}=-\frac{\bm{u}^*\bm{q}_{K_j}(1)}{\bm{u}^*\bm{q}_\mu(1) -\bm{u}^*\bm{v_N}},
	\label{eq:pderivmu}
	\end{equation}
	where  $\bm{q}_\mu$ and $\bm{q}_{K_j}$ are  solution of initial value problem
	\begin{equation}
	\begin{cases}
	\dot{\bm{q}}_{K_j}(s)=\pdv{A(s,\mu;\bK)}{K_j}\bm{q}(s)+A(s,\mu;\bK)\bm{q}_{K_j}(s)\\ \dot{\bm{q}}_\mu(s)=\pdv{A(s,\mu;\bK)}{\mu}\bm{q}(s)+A(s,\mu,\bK)\bm{q}_\mu(s)\\
	\bm{q}_{K_j}(0)={0}, \quad \bm{q}_\mu(0)={0}
	\end{cases}
	\label{eq:diffsys}
	\end{equation}
	and $\bm{v_N}$ is as in Proposition~\ref{lemma:deriv}. 
\end{proposition}

\begin{proof} Applying the well known formula for a derivative of a simple eigenvalue \cite[Lemma~2.7]{Schreiber2008}  to \eqref{wm:nleig} yields
	\begin{equation}\label{eq:demuth}
	\pdv{\mu}{K_j}=-\frac{\bm{u}^* \frac{\partial \No(\mu;\bK)}{\partial K_j} \bm{v}  } { \bm{u}^*  \frac{\partial \No(\mu;\bK)}{\partial \mu}  \bm{v}  }.
	\end{equation} 	
	The remainder of the proof is analogous to the proof of Proposition~\ref{lemma:deriv}, relying on the variational equations corresponding to \eqref{eq:nldifeq} and parameters $\mu$, respectively, $K_j$. 	
\end{proof}

The previously presented results bring us to the following algorithm for computing the gradient of a simple Floquet multiplier with respect to parameters vector $\bK$.

\begin{algo2}\textit{Computation of gradient of Floquet multiplier.}\label{algo:deFM}
	\begin{enumerate}
		\item Compute the targeted Floquet multiplier $\mu$ and its right eigenvector $\bm{v}$  by Algorithm~\ref{algo:FMsmall}. 
		\item Compute  $\bm{u}\in\C^{Nd}$ as the corresponding left eigenvector.
		\item Solve initial value problem
			\begin{equation}
			\begin{cases}
			\dot{\bm{q}}(s)=A(s,\mu;\bK) \bm q(s)\\ 
			\dot{\bm{q}}_{K_j}(s)=\pdv{A(s,\mu;\bK)}{K_j}\bm{q}(s)+A(s,\mu;\bK)\bm{q}_{K_j}(s),\ j=1,\ldots, k\\ \dot{\bm{q}}_\mu(s)=\pdv{A(s,\mu;\bK)}{\mu}\bm{q}(s)+A(s,\mu,\bK)\bm{q}_\mu(s)\\
			\bm{q}(0)=\bm{v},\quad  \bm{q}_{K_j}(0)={0},\ j=1,\ldots, k, \quad \bm{q}_\mu(0)={0}
			\end{cases}
			\label{wm:solvtot}
			\end{equation}
		on the interval $[0,\, 1]$ and compute,  by formula \eqref{eq:pderivmu}, the partial derivatives of the Floquet multiplier with respect to the elements of $\bK$.
	\end{enumerate}
\end{algo2}
If dimension $Nd$ is small, vector $\bm{u}$ can be computed as  singular vector corresponding to the smallest singular value of  matrix $\No_\delta$, obtained by solving initial value problem \eqref{eq:nldifeq} with $Nd$ independent initial vectors $\bm{v}$. For problems with high dimension $Nd$, the construction of the full matrix $\No_\delta$  for the second step of Algorithm~\ref{algo:deFM} is to be avoid.  In this situation we can separately apply  Arnoldi method and, possibly, Broyden's correction initialized with $H_0=I$,  as discussed in the previous section~\ref{sec:FMcomp},  to original system \eqref{eq:system} and to \eqref{wm:system:dual}, followed by a pairing of Floquet multipliers and  associated eigenvectors.

\begin{remark}\label{rem:consistency}
	In our implementation,  we use the same integration method with fixed step $\delta$ for numerically solving \eqref{eq:adjoint} and \eqref{wm:solvtot} as we use for \eqref{eq:nldifeq}. In particular, if the left eigenvector is computed as left singular vector  of $\No_\delta$, we  solve \eqref{eq:nldifeq} and \eqref{wm:solvtot} by either Runge-Kutta of order $4$ for non-stiff problems or trapezoidal rule for stiff problems. If the left eigenvector is computed from the discretized monodromy operator of system  \eqref{wm:system:dual}, followed by corrections using  \eqref{wm:dual}, then the initial value problems  \eqref{eq:nldifeq}, \eqref{eq:adjoint} and \eqref{wm:solvtot} are solved by trapezoidal rule, a symmetric mixed implicit-explicit scheme. The consistency between integration schemes permits to interpret the result of  Algorithm~\ref{algo:deFM} as the gradient of eigenvalue of $\No_\delta(\mu,\bK)v=0$.
\end{remark}	

\subsection{Stability optimization}\label{sec:stabopt}
The stabilization problem of an unstable system of form  \eqref{eq:systemcontrol} corresponds to finding controller parameters $\bK$ such that the Floquet multipliers are all confined to the open unit disc.
This problem, as well as the problem of increasing the 
decay rate of solutions of a stable system towards the zero equilibrium, lead us to the optimization problem of  minimizing the (squared) spectral radius of the monodromy operator as a function of controller parameters $\bK$,
\begin{equation}
\min_{\bK} \rho(\mathscr{U})^2,\ \ \ \label{eq:stabopt}
\end{equation}
where
\[
\rho(\mathscr{U})=\max_{\mu\in\sigma(\mathscr{U})} \lvert \mu\rvert.
\] 
Here $\sigma(\cdot)$ and $\rho(\cdot)$ denotes the spectrum, respectively spectral radius.

The objective function in \eqref{eq:stabopt} is in general non-convex. It may not be everywhere differentiable, even not everywhere Lipschitz continuous, although a Floquet multiplier with multiplicity one locally defines a smooth function of parameters $\bK$. The lack of smoothness  is related to the occurrence of multiple  dominant Floquet multipliers (counted with multiplicity).  In general, points of non-differentiability  occur on a set with measure zero in the parameter space, meaning that the objective function is smooth almost everywhere. Based on these properties, we use the MATLAB code HANSO (Hybrid Algorithm for Non-Smooth Optimization), described in \cite{Overton2014}, which is based on   BFGS with weak Wolfe condition in the line search.  The underlying algorithm only requires the evaluation of the objective function and its gradient with respect to the parameters, whenever it is differentiable.  This is the case if the dominant Floquet multiplier is isolated and simple, where  its gradient can be evaluated  as
\begin{equation}
\nabla_{\bK} \rho(\mathscr{U})^2=2\Re\left(\bar\mu_D \nabla_{\bK} \mu_D\right),
\end{equation}
where $\mu_D$ denotes the dominant Floquet multiplier, and its gradient, $\nabla_{\bK} \mu_D$, can be evaluated as described in the previous subsection.

The  accuracy of the results depends on the resolution of solving initial value problems 
\eqref{eq:nldifeq}, 
and \eqref{wm:solvtot}. To avoid fluctuations of the objective function and convergence problems, we always used the same step-size $\delta$ for these initial value problems, and kept it fixed during the optimization process. When using the time-integration schemes described in Remark~\ref{rem:consistency}, the overall method can be interpreted as stability optimization  of discretized eigenvalue problem
$\No_{\delta}(\mu;\bK)\bm{v}=0$.

\section{Numerical examples}\label{sec:examples}

\subsection{Stability and stabilization of a scalar system}\label{sec:example1}
For scalar systems of the form \eqref{eq:commensurate}, the nonlinear eigenvalue problem has an explicit expression, \eqref{eq:d1N1}. This explicit expression permits to compute the derivative of the Floquet multiplier with respect to an element $K_j$ of $\bK$ by \eqref{eq:demuth},   as follows
\begin{equation}
\pdv{\mu}{K_j}=\frac{\sum_{j=0}^h\mu^{1-j}\pdv{a_j}{K_j}}{1+\sum_{j=1}^hj\mu^{-j}a_j}.\label{eq:ded1N1}
\end{equation}
As a first example, we test the proposed methods on a scalar system with $N=1$, so that the results can be compared with the explicit expressions \eqref{eq:d1N1} and \eqref{eq:ded1N1}. We consider the following system 
\begin{equation}\label{eq:ex1_scalar}
\dot{x}(t)=(K\cos(2t))x(t)+(\sin(2t)+K)x(t-\pi)+0.1\cos(2t)\exp(\sin(2t))x(t-2\pi).
\end{equation} 
The Floquet multipliers and their derivatives with respect to controller parameter $K$ can be obtained by \eqref{eq:d1N1} and \eqref{eq:ded1N1}, with $a_0=a_2=0$ and $a_1=K\pi$,
\begin{equation}
\mu_k=\begin{cases}
\frac{K\pi}{W_k({K\pi})}, \quad &\text{if } K\neq0,\\
1, \quad &\text{if } K=0,
\end{cases}\qquad \pdv{\mu}{K}=
\begin{cases}
\frac{\pi}{1+W_k({K\pi})}, \quad &\text{if } K\neq0,\\
\pi, \quad &\text{if } K=0,
\end{cases}
\label{eq:ex1FloMu}
\end{equation}
where $W_k$ is the $k$-th branch of the Lambert $W$ function.
For $K=\nicefrac{e}{\pi}$, the largest Floquet multiplier in modulus satisfies  $\mu=e$ and its derivative $\pdv{\mu}{K}=\nicefrac{\pi}{2}$.  For this eigenvalue, the left pane of Figure~\ref{fig:conv1} depicts the relative error induced by the discretization of the monodromy operator (the first stage of Algorithm~\ref{algo:FMsmall}), as a function of the number of collocation points $M$, and illustrates the expected spectral convergence.    In the right pane, we depict the approximation error on the eigenvalue and its derivative,  induced by using a Runge-Kutta method of order $4$ for solving initial value problems \eqref{eq:nldifeq} and \eqref{wm:solvtot}, as a function of step-size $\delta$ 
(recall that the second stage of Algorithm~\ref{algo:FMsmall}  can be interpreted in terms of approximate eigenvalue problem $\No_{\delta}(\mu)\bm{v}=0$).

\begin{figure}[!h]
	\begin{subfigure}[h]{0.47\linewidth}
		\begin{tikzpicture}
		
		\begin{axis}[%
		width=0.7\textwidth,
		height=0.5\textwidth,
		at={(0\textwidth,0\textwidth)},
		scale only axis,
		axis x line*={bottom, color=gray!125}, 						
		axis y line=left,										
		ytick style={color=gray!125},
		xtick style={color=gray!125},
		tick align=outside,
		xmode=log,
		xmin=2,
		xmax=100,
		xtick={  1,  10, 100},
		xticklabels={\tiny{1},\tiny{10},\tiny{100}},
		extra x ticks ={2,5,15,50}, extra x tick labels={\tiny{2},\tiny{5},\tiny{15},\tiny{50}},
		xminorticks=true,
		xlabel={$M$-degree spline},
		ymode=log,
		ymin=1e-16,
		ymax=20,
		ytick={1e-16,1e-15,1e-14,1e-13,1e-12,1e-11,1e-10,1e-09,1e-08,1e-07,1e-06,1e-05,1e-04,1e-03,1e-02,1e-01,1},
		yticklabels={\tiny{$10^{-16}$},\tiny{$10^{-15}$}, , , , , \tiny{$10^{-10}$},, , , , , ,,, ,\tiny{$1$}},
		extra y ticks ={6.08408543232763e-07}, extra y tick labels={\tiny{$6\cdot10^{-7}$}},
		axis background/.style={fill=white},
		ylabel style = {font=\small, yshift=-0.5cm, rotate=-90, color=black,at={(0.04\textwidth,-0.05\textwidth)}},
		ylabel={$\frac{\lvert \mu-\mu_M \rvert}{\lvert \mu\rvert}$}, %
		xlabel style={font=\small, yshift=0.3cm, color=black},
		title style={font=\bfseries, yshift=-0.3cm,color=black},
		title={ \small{Error spline discretization}},
		]
		\addplot [color=black!50!white,line width=1.1pt,only marks,mark=o,mark size=2 pt,forget plot]
		table[row sep=crcr]{%
			15	6.08408543232763e-07\\
		};
		
		\addplot [color=black,solid,line width=1pt,mark=o,mark size=1 pt,forget plot]
		table[row sep=crcr]{%
			2	4.65402077185175\\
			3	0.219263122394391\\
			4	0.229761170633599\\
			5	0.0909429436914303\\
			6	0.00273921998623157\\
			7	0.00361190123447791\\
			8	0.00207174431738241\\
			9	0.00034792557704186\\
			10	1.68364310617957e-05\\
			11	1.48589769168743e-05\\
			12	7.51346220544696e-06\\
			13	4.7916530249881e-06\\
			14	4.97022132090138e-07\\
			15	6.08408543232763e-07\\
			16	2.13670028192679e-07\\
			17	2.35227718178089e-08\\
			18	2.82133719859768e-08\\
			19	9.44599469705878e-09\\
			20	1.31371755709072e-11\\
			21	2.57705549251978e-09\\
			22	9.43051624752587e-10\\
			23	2.46154137110853e-10\\
			24	2.38036871178527e-10\\
			25	4.41429226525452e-11\\
			26	2.65744642021872e-11\\
			27	2.42421756611519e-11\\
			28	2.38031970039817e-12\\
			29	5.58795161512827e-12\\
			30	1.84527872450222e-12\\
			31	7.26185385605343e-13\\
			32	4.70345944917386e-13\\
			33	1.63371290349908e-16\\
			34	6.74723429145122e-14\\
			35	2.33620945200369e-14\\
			36	2.94068322629835e-15\\
			37	7.18833677539597e-15\\
			38	3.26742580699817e-16\\
			39	4.90113871049725e-16\\
			40	1.30697032279927e-15\\
			41	4.90113871049725e-16\\
			42	3.26742580699817e-16\\
			43	1.47034161314918e-15\\
			44	1.30697032279927e-15\\
			45	8.16856451749542e-16\\
			46	8.16856451749542e-16\\
			47	1.30697032279927e-15\\
			48	9.8022774209945e-16\\
			49	4.90113871049725e-16\\
			50	1.63371290349908e-15\\
			51	1.30697032279927e-15\\
			52	3.9209109683978e-15\\
			53	3.26742580699817e-16\\
			54	1.63371290349908e-15\\
			55	3.26742580699817e-15\\
			56	3.10405451664826e-15\\
			57	4.90113871049725e-16\\
			58	1.79708419384899e-15\\
			59	1.47034161314918e-15\\
			60	0\\
			61	1.63371290349908e-15\\
			62	3.26742580699817e-15\\
			63	2.12382677454881e-15\\
			64	2.12382677454881e-15\\
			65	3.26742580699817e-16\\
			66	2.77731193594844e-15\\
			67	3.26742580699817e-16\\
			68	2.77731193594844e-15\\
			69	3.43079709734808e-15\\
			70	1.14359903244936e-15\\
			71	3.26742580699817e-16\\
			72	1.63371290349908e-16\\
			73	3.9209109683978e-15\\
			74	1.63371290349908e-16\\
			75	8.16856451749542e-16\\
			76	1.47034161314918e-15\\
			77	3.43079709734808e-15\\
			78	3.43079709734808e-15\\
			79	2.61394064559853e-15\\
			80	1.79708419384899e-15\\
			81	1.14359903244936e-15\\
			82	3.59416838769798e-15\\
			83	1.14359903244936e-15\\
			84	2.94068322629835e-15\\
			85	1.9604554841989e-15\\
			86	2.61394064559853e-15\\
			87	2.94068322629835e-15\\
			88	1.14359903244936e-15\\
			89	3.9209109683978e-15\\
			90	9.8022774209945e-16\\
			91	2.28719806489872e-15\\
			92	1.9604554841989e-15\\
			93	3.26742580699817e-15\\
			94	2.77731193594844e-15\\
			95	1.63371290349908e-16\\
			96	2.12382677454881e-15\\
			97	0\\
			98	6.53485161399634e-16\\
			99	4.90113871049725e-16\\
			100	2.28719806489872e-15\\
		};
		
		\end{axis}
		\end{tikzpicture}%
	\end{subfigure}\qquad 
	\begin{subfigure}[h]{0.47\linewidth}
		\begin{tikzpicture}
		
		\begin{axis}[%
		width=0.7\textwidth,
		height=0.5\textwidth,
		at={(0\textwidth,0\textwidth)},
		scale only axis,
		axis x line*={bottom, color=gray!125}, 						
		axis y line*=left,										
		ytick style={color=gray!125},
		xtick style={color=gray!125},
		tick align=outside,
		xmode=log,
		xmin=1e-4,
		xmax=1e-1,
		xtick={ 1e-4, 1e-3,1e-2,1e-1},
		xticklabels={\tiny{$10^{-4}$},\tiny{$10^{-3}$} ,\tiny{$0.01$} ,\tiny{$0.1$}},
		xminorticks=true,
		xlabel={Step-size $\delta$},
		ymode=log,
		ymin=1e-14,
		ymax=1e-3,
		ytick={1e-14,1e-13,1e-12,1e-11,1e-10,1e-09,1e-08,1e-07,1e-06,1e-05,1e-04,1e-03},
		yticklabels={\tiny{$10^{-14}$} ,\tiny{$10^{-13}$} ,\tiny{$10^{-12}$} , , \tiny{$10^{-10}$}, , ,\tiny{$10^{-7}$} , , ,\tiny{$10^{-4}$} ,\tiny{$10^{-3}$}},
		axis background/.style={fill=white},
		axis background/.style={fill=white},
		xlabel style = {font=\small, yshift=0.1cm, color=black},
		title style={font=\bfseries, yshift=-0.2cm,color=black},
		title={ \small{Errors for Algorithms~\ref{algo:FMsmall}-\ref{algo:deFM}}},
		]
		\addplot [color=black,solid,line width=1.2pt,mark=o,mark size=0.4 pt,forget plot]
		table[row sep=crcr]{%
			0.1	0.000131632409389647\\
			0.0909090909090909	8.77724283025066e-05\\
			0.0833333333333333	6.06318166671099e-05\\
			0.0714285714285714	3.15023780048363e-05\\
			0.0666666666666667	2.35099525555807e-05\\
			0.0588235294117647	1.38377192860151e-05\\
			0.0526315789473684	8.64937857210721e-06\\
			0.0454545454545455	4.66409934716895e-06\\
			0.04	2.72687898107557e-06\\
			0.0357142857142857	1.69683843482319e-06\\
			0.032258064516129	1.10946519386975e-06\\
			0.0285714285714286	6.69450259371784e-07\\
			0.0256410256410256	4.27222413209292e-07\\
			0.0227272727272727	2.59309684600216e-07\\
			0.02	1.53003324152433e-07\\
			0.0178571428571429	9.5968086002101e-08\\
			0.0158730158730159	5.9170304706286e-08\\
			0.0142857142857143	3.84280344854983e-08\\
			0.0126582278481013	2.34370884771591e-08\\
			0.0112359550561798	1.4412191612913e-08\\
			0.01	8.96702454509511e-09\\
			0.00892857142857143	5.65683417902108e-09\\
			0.008	3.6223795054869e-09\\
			0.00709219858156028	2.22319099042382e-09\\
			0.00632911392405063	1.40236821048716e-09\\
			0.00564971751412429	8.86092877857252e-10\\
			0.0050251256281407	5.52057239384925e-10\\
			0.00448430493273543	3.48706053603168e-10\\
			0.00398406374501992	2.16468103312661e-10\\
			0.00355871886120996	1.37374016916527e-10\\
			0.00316455696202532	8.56366164607764e-11\\
			0.00282485875706215	5.42364910842336e-11\\
			0.00251256281407035	3.38820620195386e-11\\
			0.00224215246636771	2.14299022690685e-11\\
			0.00199600798403194	1.34137631654696e-11\\
			0.00177935943060498	8.44956313689726e-12\\
			0.00158730158730159	5.36200912057434e-12\\
			0.00141442715700141	3.34568065507577e-12\\
			0.00125944584382872	2.09687051164107e-12\\
			0.00112233445566779	1.340461437321e-12\\
			0.001	8.3891157594678e-13\\
			0.00089126559714795	5.51051362350241e-13\\
			0.000794912559618442	3.9225446813013e-13\\
			0.000708215297450425	1.72193340028803e-13\\
			0.000631313131313131	9.55722048546964e-14\\
			0.000562429696287964	1.2563252227908e-13\\
			0.00050125313283208	1.51935300025415e-14\\
			0.000446827524575514	8.65867838854515e-14\\
			0.000398247710075667	3.25108867796318e-14\\
			0.00035486160397445	4.54172187172745e-14\\
			0.000316255534471853	9.81861455002949e-14\\
			0.00028184892897407	1.0684482388884e-13\\
			0.000251193167545843	2.95702035533334e-14\\
			0.000223914017017465	1.48667874218417e-13\\
			0.000199560965875075	1.46380676153518e-13\\
			0.000177841010136938	1.57163181316612e-13\\
			0.000158503724837534	1.22691839052781e-13\\
			0.000141262890238734	1.17954071632634e-13\\
			0.000125897016240715	3.61377294253997e-13\\
			0.000112208258527828	2.58126638752855e-13\\
			0.0001	1.52262042606115e-13\\
		};
		\addplot [color=black!70!white,solid,line width=1.2pt,mark=o,mark size=0.4 pt,forget plot]
		table[row sep=crcr]{%
			0.1	0.000747082142019412\\
			0.0909090909090909	0.000505677036993882\\
			0.0833333333333333	0.000353915905406502\\
			0.0714285714285714	0.000187937011634191\\
			0.0666666666666667	0.000141550949246493\\
			0.0588235294117647	8.46419513749462e-05\\
			0.0526315789473684	5.36039456592009e-05\\
			0.0454545454545455	2.93707033834738e-05\\
			0.04	1.73921131899647e-05\\
			0.0357142857142857	1.09360303509473e-05\\
			0.032258064516129	7.21295332777994e-06\\
			0.0285714285714286	4.39422066216061e-06\\
			0.0256410256410256	2.82638351732904e-06\\
			0.0227272727272727	1.72936549050703e-06\\
			0.02	1.02832229025864e-06\\
			0.0178571428571429	6.49025707071195e-07\\
			0.0158730158730159	4.02531390068632e-07\\
			0.0142857142857143	2.62683205899768e-07\\
			0.0126582278481013	1.61015141832545e-07\\
			0.0112359550561798	9.94557428888024e-08\\
			0.01	6.21235935806633e-08\\
			0.00892857142857143	3.93262177144427e-08\\
			0.008	2.52591478069511e-08\\
			0.00709219858156028	1.55489809440543e-08\\
			0.00632911392405063	9.83312442108043e-09\\
			0.00564971751412429	6.22725296267692e-09\\
			0.0050251256281407	3.88799415062258e-09\\
			0.00448430493273543	2.46038126086578e-09\\
			0.00398406374501992	1.52994603509001e-09\\
			0.00355871886120996	9.72343627838584e-10\\
			0.00316455696202532	6.0701607021459e-10\\
			0.00282485875706215	3.84906062737182e-10\\
			0.00251256281407035	2.40621966587413e-10\\
			0.00224215246636771	1.52401861202765e-10\\
			0.00199600798403194	9.55857045949735e-11\\
			0.00177935943060498	6.03003482428456e-11\\
			0.00158730158730159	3.81919592570281e-11\\
			0.00141442715700141	2.40109261172761e-11\\
			0.00125944584382872	1.50792217817973e-11\\
			0.00112233445566779	9.54053318050381e-12\\
			0.001	5.99838477125435e-12\\
			0.00089126559714795	3.81270759415262e-12\\
			0.000794912559618442	2.49906783171527e-12\\
			0.000708215297450425	1.44439589934197e-12\\
			0.000631313131313131	8.69634328905047e-13\\
			0.000562429696287964	6.84596725436792e-13\\
			0.00050125313283208	3.31343118815577e-13\\
			0.000446827524575514	3.15228308429495e-13\\
			0.000398247710075667	3.88734461067763e-14\\
			0.00035486160397445	2.74234492535077e-14\\
			0.000316255534471853	1.14075894575158e-13\\
			0.00028184892897407	2.04686363500408e-13\\
			0.000251193167545843	6.00771439831998e-14\\
			0.000223914017017465	3.1395608655691e-13\\
			0.000199560965875075	2.50062276955954e-13\\
			0.000177841010136938	3.4349990559806e-13\\
			0.000158503724837534	1.62561683719246e-13\\
			0.000141262890238734	2.91621524793744e-13\\
			0.000125897016240715	6.5448747445227e-13\\
			0.000112208258527828	4.18702354066442e-13\\
			0.0001	3.63431381601898e-13\\
		};
		
		\pgfplotsset{
			after end axis/.code={  
				\node[right,align=right] at (axis cs: 	0.01,	0.000000001) {\textcolor{black}{\small{$\frac{\lvert \mu-\mu_\delta \rvert}{\lvert \mu\rvert}$}}};  
				\node[left,align=left] at (axis cs: 	0.01,	0.000002) {\textcolor{black!70!white}{\small{$\frac{\left\lvert \nicefrac{\partial \mu}{\partial K}-\nicefrac{\partial \mu_\delta}{\partial K}\right\rvert}{\left\lvert \nicefrac{\partial \mu}{\partial K}\right\rvert}$}}};  
			}}
			\end{axis}
			\end{tikzpicture}%
	\end{subfigure}
	\caption{Approximation errors on the dominant Floquet multiplier of system \eqref{eq:ex1_scalar} with $K=\nicefrac{e}{\pi}$. Error induced by approximating $\mathscr{U}$ by $\mathscr{U}_M$ (left pane). Error induced by approximating $\No(\mu)$ by $\No_{\delta}(\mu)$ (right pane).}
		\label{fig:conv1}
\end{figure}
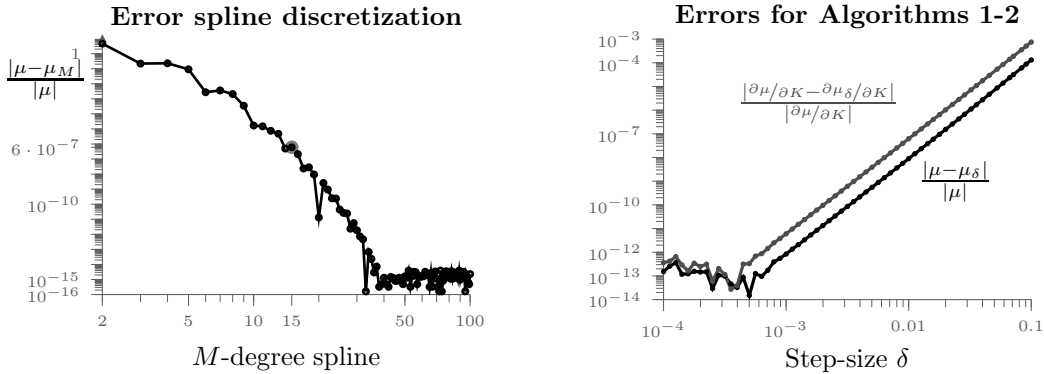

System \eqref{eq:ex1_scalar} is unstable for $K=\nicefrac{e}{\pi}$, since the dominant Floquet multiplier $\mu=e$ does not lies within the unit circle. 
Selecting $M=15$ collocation points in the first stage of Algorithm~\ref{algo:FMsmall}, followed by applying Broyden's method to \eqref{wm:Ndelta} with a Runge-Kutta method of order $4$ and step-size $\delta=2\cdot10^{-4}$  to compute dominant Floquet multipliers, an application of the stability optimization approach of section~\ref{sec:stabopt} yields the stabilizing controller parameter $K= -0.1295$, for which the the spectral radius is $\rho(\mathscr{U}) =0.3935$. Figure~\ref{fig:optimization} displays the iterates generated by the 
optimization routine. The final controller parameter found is very close to the minimizer  $K=-(e\pi)^{-1}$, for which the dominant Floquet multiplier $\mu=e^{-1}$ corresponds to a double  non-semi-simple eigenvalue.
Finally, Figure~\ref{fig:spectra} shows the eigenvalues of \eqref{eq:Mfineig} in the first and last iteration of the stability optimization procedure.

\begin{figure}[!h]
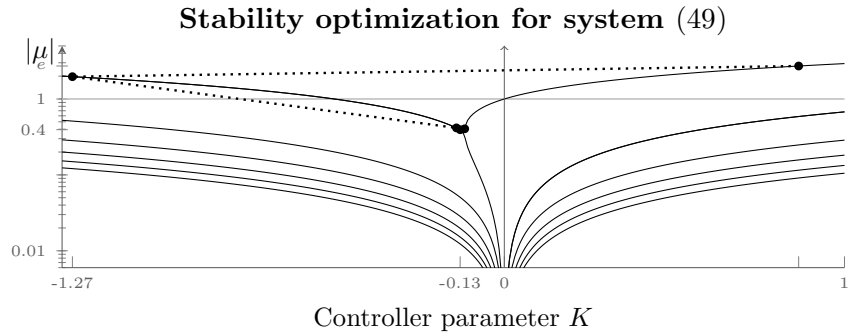

	{\centering
\par}
	\caption{The stability optimization path is compared with the exact modulus of the $10$ dominant Floquet multipliers of system \eqref{eq:ex1_scalar} varying the controller parameter $K$. The stability optimization provides a stabilizing controller already at the third iteration. 
	} 
	\label{fig:optimization}
\end{figure}

\begin{figure}[!h]
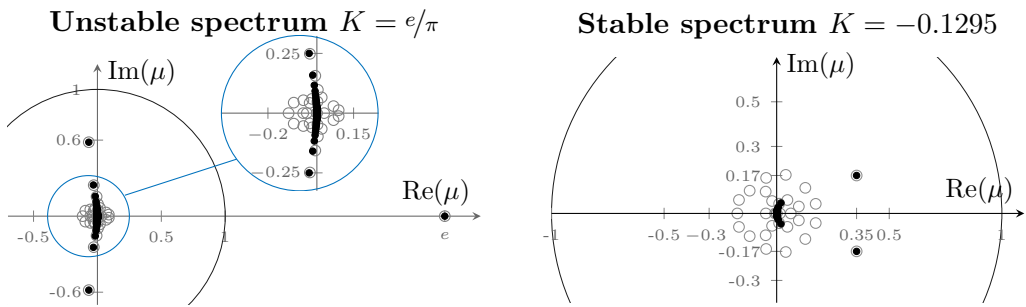

	\begin{subfigure}[h]{0.47\linewidth}
		{\centering
			%
\par}
	\end{subfigure}
	\caption{
		The left and right panes show the Floquet multiplier approximations obtained as eigenvalues of \eqref{eq:Mfineig} with $M=15$ (empty gray dot)  and their analytical value computed by \eqref{eq:ex1FloMu} (full black dot) at the beginning and at the end of the stability optimization process, respectively. }
	\label{fig:spectra}
\end{figure}

\subsection{Delayed Mathieu's equation}
As second example, we consider a variant of Mathieu's equation with delayed input and PID controller,
\[
\begin{array}{l}
\ddot{z}(t)+(\nu+\epsilon \cos 2 t)z(t)=-u(t-\tau),
\\
u(t)=k_i\int_{0}^tz(t)\dd t  +k_pz(t)+k_d\dot{z}(t), 
\end{array}
\]
where $k_i$, $k_p$ and $k_d$ are respectively integral (I), proportional (P)  and 
derivative (D) gain. The Mathieu's equation with delayed PI and PD feedback controllers is analyzed in \cite{Sheng2005}, and arises as model for the stick balancing problem in \cite[Section~5.4]{Insperger2011}.

The Floquet multipliers of the Mathieu's equation with delayed PID controller can be 
inferred from the spectrum of the following linear periodic time-delay system 
\begin{equation}
\dot{x}(t)=%
x(t-\tau).\label{eq:Mathieu}
\end{equation}
If $k_i=0$, Mathieu's equation with PD feedback controller can be rewritten as a linear periodic time-delay system of dimension $d=2$, which presents the same spectrum of system  \eqref{eq:Mathieu}, except for a non-physical Floquet multiplier of \eqref{eq:Mathieu} equal to one.

For $\nu=4$ and $\epsilon=2$ the system without controller is unstable. 
Therefore, we design PI, PD, and PID controllers in the presence of input delay $\tau=\nicefrac{3\pi}{4}$. Table~\ref{tab:Mathieu} summarizes the results of the stability optimization, with $M=10$ collocation points and a Runge-Kutta method of order $4$ with step-size $\delta=0.001$.  We used a direct method to solve discretized eigenvalue problem \eqref{eq:Mfineig} and obtained left eigenvector by explicitly construct matrix $\No_{\delta}(\mu)$. 
As expected, increasing the number of controller parameters result in smaller spectral radii.

\begin{table}[!h]
	\caption{Results of the stability optimization algorithm for Mathieu's equation with delayed feedback controller \eqref{eq:Mathieu}, where $\nu=4$, $\epsilon=2$, $\tau=\nicefrac{3\pi}{4}$. All the designed controllers stabilize the system, since the spectral radii $\rho(\mathscr{U})$ are smaller than $1$.}
	\label{tab:Mathieu}
	\begin{center}
		\begin{tabular}{lcccc}
			\toprule
			& $\rho(\mathscr{U})$ 	& $k_i$ & $k_p$ & $k_d$ \\
			\midrule
			PI controller 	&0.5339 & 0.3215 & 0.7541 &        \\ 
			PD controller  	&0.2858 & 		  & 0.7012	& 0.0231\\ 
			PID controller  &0.1592& 1.4131  & 0.9666 & 0.3787 \\ 
			\bottomrule
		\end{tabular}
	\end{center}
\end{table}

\subsection{A PDE model for a milling problem in machining}\label{sec:example3}
The last numerical example consider a  milling model, described in \cite{Rott2010}. Some Floquet multipliers of this system are computed by deflated Broyden's method in \cite{Jarlebring2018}.
This milling model describes the interaction between a rotating cutter and a visco-elastic workpiece by a periodic cutting-force $f$. The rotating cutter is attached to a spring, and  modeled by
\begin{equation*}
\ddot{q}(t)+2K\dot{q}(t)+q(t)=-f(t),
\end{equation*}
where $K$ is the damping parameter to be designed in order to improve the stability properties of the overall system. 
The workpiece is modeled by Kelvin-Voigt material leading to the following partial differential equation (PDE),
\begin{align*}
&\frac{\partial u(t,x)}{\partial t^2}- \frac{ \partial u(t,x)}{\partial x^2}-\frac{\partial u(t,x)}{\partial x^2\partial t}=0,\quad x\in[0,\,1],\\
&u(t,0)\equiv0,\  \frac{ \partial u(t,1)}{\partial x}+
\frac{\partial u(t,1)}{\partial t\partial x}=-f(t). 
\end{align*}
This PDE is discretized in space using $n$ linear finite elements defined on an equidistant grid. By applying Galerkin method, we get a system of ordinary differential equations,
\begin{equation*}
P_n\ddot{U}(t)+D_nU(t)+D_n\dot{U}(t)=-f(t)e_n,
\end{equation*}
where $U(t)\in\R^n$ and 
\[
P_n=\frac{1}{6n}
\begin{pmatrix}
4 & 1  &         &          \\
1 & \ddots  & \ddots     &           \\
&\ddots&   4 &  1 \\
&      &  1    &  2
\end{pmatrix}, \ 
D_n=n
\begin{pmatrix}
2 &  -1 &         &          \\
-1 & \ddots  & \ddots     &           \\
&\ddots&  2 &  -1 \\
&      & -1  &  1  
\end{pmatrix}, \
e_n=\begin{pmatrix} 0\\ \vdots \\0\\ 1 \end{pmatrix}.
\]
The periodic cutting-force couples the rotating cutter and workpiece dynamics, and is influenced by the previous cut, occurring with  delay $\tau=1$; it is modeled by 
\begin{equation*}
f(t)=w(t)\left(q(t)-q(t-\tau)\right)+w(t)\left(u(t,L)-u(t-\tau,L)\right),
\end{equation*}
where $u(t,L)$ is discretized by $e_n^\transpose U(t)$, and $w(t)$ is a function 
with period $1$, defined on the interval $[0,\, 1]$ as
\begin{equation}\label{eq:wperiodic}
w(t)=\begin{cases}
\sin^2(2\pi t)+0.5 \sin(4\pi t),  & t\in [0,\, \nicefrac{1}{2}], \\
0, & t\in (\nicefrac{1}{2},\, 1].
\end{cases}
\end{equation}
Setting $y(t)=\left(U^\transpose\,q\,\dot{U}^\transpose\,\dot{q}\right)^\transpose$, we can rewrite the milling model as a periodic linear time-delay system, whose dimension $d=2(n+1)$ depends on the PDE discretization,
\begin{equation}\label{eq:milling}
E\dot{y}(t)=(A(K)-Fw(t))y(t)+Fw(t)y(t-\tau),
\end{equation}
where 
\[
E=\begin{pmatrix} 
I_n &  & &\\ 
& 1 &  &\\ 
& &  P_n  &\\ 
& &  &1
\end{pmatrix},\, 
A(K)=-\begin{pmatrix} 
&  & -I_n & \\ 
&  &  &-1\\ 
 D_n &  &  D_n &\\ 
& 1 &  &2K
\end{pmatrix},\, 
F=\begin{pmatrix} 
&  &  &\\ 
&  & &\\ 
e_ne_n^\transpose & e_n & &\\ 
e_n^\transpose  &1&  &
\end{pmatrix}.
\]

System \eqref{eq:milling} can be reformulated into \eqref{eq:system} by pre-multiplication with matrix $E^{-1}$. However, as this reduces the sparsity of the matrices, we rely instead on a slight extension of the presented results to models with a constant non-singular matrix leading coefficient matrix $E$, as described in Appendix~\ref{appendix:mass}.

Let us consider a PDE discretization with $n=250$, leading to a system dimension $d=502$.  
The largest Floquet multipliers in modulus are iteratively approximated by Arnoldi method, and their accuracy is then refined by Broyden's method, as explained in section~\ref{sec:FMcomp}. To compute a left eigenvector, appearing in the expression for the derivative of the Floquet multiplier with respect to $K$, we separately apply to the 
Arnoldi and Broyden's methods to the model \eqref{eq:milling} and corresponding model \eqref{wm:system:dual},   followed by a pairing of Floquet multiplier approximations (see section~\ref{sec:deNLeig}). For the stiffness of the problem and the consistency (see Remark~\ref{rem:consistency}),  the initial value problems \eqref{eq:nldifeq}, \eqref{wm:solvtot} and \eqref{eq:adjoint} are solved by trapezoidal rule with step-size $\delta=0.01$.

The milling model \eqref{eq:milling}  is periodic with $T=\tau_1=1$. However, differently from the previous examples, the periodic system matrices are not smooth, presenting a discontinuity at $t=\nicefrac{1}{2}+kT$ with $k\in\Z$, due to 
\eqref{eq:wperiodic}. As a consequence, the solutions exhibit a discontinuity in the derivative at these time-instants. The presence of such a  discontinuity decreases the accuracy of the discretization by the collocation approach if it occurs in an interior point of the domain of one of the  polynomials \eqref{wm:parsol}, defining the spline approximation.   This can be overcome by taking $\Delta=\nicefrac{1}{2}\;n$,\, $n\in\N$ so the  discontinuity is located at the endpoints of the support of two $M$-degree polynomials. 

For $n_1=N=2$, the  system \eqref{eq:milling} requires a high number of collocation points $M=250$,  to 
reliably approximate the 
dominant Floquet multipliers by Arnoldi method. Therefore, we increase the sparsity of the matrices $S_q$ and $S_\phi$ used in the Arnoldi iteration \eqref{eq:Arnoldi} by setting $M=20$ and $n_1=N=26$, which is 
beneficial in the employed sparse LU factorization of $S_q$. 
Indeed, the spline approximation of the eigenfunction is achieved by 
$26$ polynomials of degree $M=20$, instead of $2$ polynomials of degree $250$. The resulting $S_q$ and $S_\phi$  matrices in \eqref{eq:Arnoldi} present a block diagonal structure where each of the $26$ blocks has dimension $d\times 20$.

\begin{figure}[h]
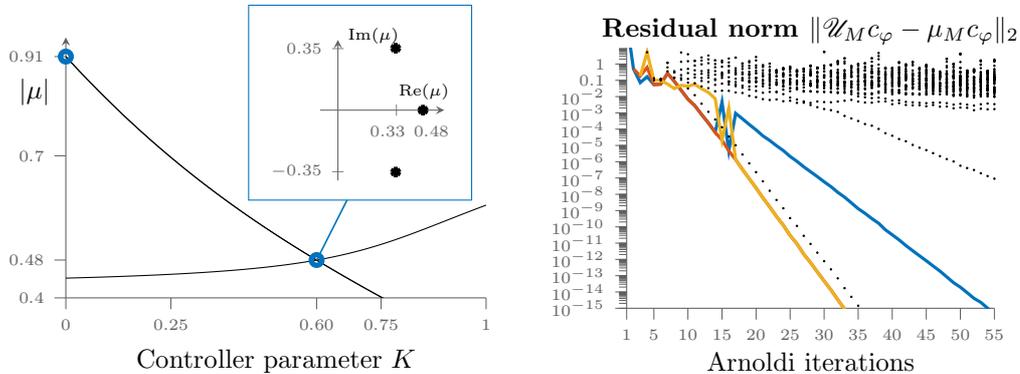

	\begin{subfigure}[h]{0.47\linewidth}
		{\centering
\par}
	\end{subfigure}
	\caption{Modulus of the dominant Floquet multipliers as a function of $K$, and the three dominant Floquet multipliers corresponding to the minimum spectral radius.  The initial and final iterations of the optimization process are indicated by blue dots (left pane). Convergence history for Arnoldi method applied to $\mathscr{U}_M$ for $M=20$, $N=26$, and $K=0.5968$, close to the minimum of the spectral radius (right pane).  } 
	\label{fig:milling}
\end{figure}

The stability optimization almost halves the spectral radius, 
  from $\rho(\mathscr{U})=0.9095$ at initial gain $K=0$ to $\rho(\mathscr{U})=0.4799$ for $K=0.5968$. The designed controller is close to a minimizer of the spectral radius, as illustrated in the left pane of Figure~\ref{fig:milling}. The minimum is characterized by a real and  a conjugate pair of Floquet multipliers 
  with the same modulus. The approximations of the largest Floquet multipliers requires few Arnoldi iteration to converge, for example  for the final controller gain the residual norm of the approximated conjugate pair of Floquet multipliers reaches machine precision within $35$ iterations, compared to $55$  for the dominant real Floquet multiplier, right pane of Figure~\ref{fig:milling}.

\section{Concluding remarks}\label{sec:conclusion}

Central in the paper is the the dual interpretation of Floquet multipliers as solutions of either an operator eigenvalue problem or a finite-dimensional nonlinear eigenvalue problem.
We demonstrated that the related one-to-one mapping can not only be expressed in terms of right eigenpairs (Theorem~\ref{th:eigenpair}), but also in terms of left eigenpairs (Theorem~\ref{th:dual}), where surprisingly the time-shifts in \eqref{wm:system:dual} depend on the coefficient matrices.
The dual interpretation can be exploited from a computational point of view, as the nonlinear eigenvalue problem formulation lays at the basis of local corrections to improve the accuracy of (multiple) Floquet multiplier estimates obtained by discretizing the monodromy operator. This results in a robust method for accurately computing dominant Floquet multipliers. In addition, computationally tractable expressions for derivatives can be obtained from the nonlinear eigenvalue problem formulation that are useful in the context of stability optimization.  It should be noted that except for Theorem~\ref{th:dual}, the derived results can be directly generalized to the computation of eigenvalues of the solution operator if the integration time is different from $T$ or if the matrices are not periodically varying.

As we have seen, the eigenvalue problem for the discretized monodromy operator $\mathscr{U}_M$ can also be obtained from a spectral discretization of boundary value problem~\eqref{eq:difeq}. The dependence of matrix  $A(s,\mu)$  on $\mu$ in the ordinary differential equation is in terms of negative integer powers, suggesting a smoothing effect of increasing $|\mu|$ on its solutions. Further research includes clarifying whether this implies that outside a given circle in the complex plane that contains no Floquet multipliers, the number of Floquet multipliers and the number of eigenvalues of $\mathscr{U}_M$ match for sufficiently large $M$, a property which has  been observed in our numerical experiments. 

Finally, Assumption~\ref{as1} describes the most general situation where Floquet multipliers can be related to a boundary value problem in terms of an \emph{ordinary  differential equation}.  If the assumption is not satisfied or if the minimum $N$ is prohibitively large, 
stability analysis and stabilization can still be performed on (parametrized) matrix $\mathscr{U}_M$, obtained by the spectral discretization, resulting in ``discretize-first" approach.

\subsection*{Funding} This work was supported by the project C14/17/072 of the KU Leuven Research Council, by the project G0A5317N of the Research Foundation-Flanders (FWO - Vlaanderen), and by the project UCoCoS, funded by the European Unions Horizon 2020 research and innovation programme under the Marie Sklodowska-Curie Grant Agreement No 675080.

\bibliography{Arxiv_PTDS}

\appendix
\section{Proof of Lemma~\ref{th:operator}} \label{ap1}
$\Rightarrow$\ \ Let $(\mu,\phi)$ be an 
eigenpair of $\mathscr{U}$. Consider the solution of initial value problem \eqref{eq:ivp}, initialized with $\phi$ at $t_0=0$, in the interval $[-\tau_h,\,T]$, \ie $x(t;0,\phi)$ for  $t\in[-\tau_h,\,T]$. By assumption we have $x(T+\theta;0,\phi)=\mu\; \phi(\theta)$ for every $\theta\in[-\tau_h,\,0]$. 
We divide the interval $[-\tau_h,\,T]$ into $\Delta$-length subintervals and define for $n=-n_h,\ldots, N$
\begin{equation*}
q_n(s)=x\left((s+n-1)\Delta;0,\phi\right), \quad s\in[0,\,1].
\end{equation*}
We define $\bm{q}(s)=\big(q_1^\transpose(s)\, \cdots\,  q_N^\transpose(s)\big)^\transpose$.
Since the solution is initialized on an eigenfunction, i.e. $\phi=\frac{1}{\mu}x_T(\cdot;0,\phi)$, functions $\{q_n(s)\}_{n=-n_h+1}^0$ are related to $\{q_n(s)\}_{n=1}^N$ 
\begin{equation}
q_n(s)=\mu^{-1}q_{N+n}(s)=\mu^{\left\lfloor \frac{n-1}{N}\right\rfloor} q_{(n-1)\bmod_N+1}(s), \ \  n=-n_h+1,\ldots,0.\label{eq:nshift}
\end{equation}
Furthermore, as illustrated in Figure~\ref{fig:th}, we have
\begin{equation*}
q_n(s-\tau_j)=x\left((s+n-n_j-1)\Delta;0,\phi\right)
=q_{n-n_j}(s)
=\mu^{a_{n-n_j}}q_{b_{n-n_j}}(s), \ n\geq 1.
\end{equation*}

From the chain rule, it follows that for $n=1,\ldots,N$,
\begin{align*}
\begin{array}{rcl}
\dot{q}_n(s)&=&\Delta\; \dot{x}((s+n-1)\Delta;0,\phi) \\[2pt]
&=&\Delta  \sum_{j=0}^hA_j( (s+n-1)\Delta  )x( (s+n-n_j-1)\Delta;0,\phi)\\[2pt]
&=&\Delta  \sum_{j=0}^hA_j( (s+n-1)\Delta  )q_{n-n_j}(s) \\[2pt]
&=&\Delta  \sum_{j=0}^hA_j( (s+n-1)\Delta  ) \mu^{a_{n-n_j}}q_{b_{n-n_j}}(s),
\end{array}
\end{align*}	
hence, differential equation in \eqref{eq:difeq} is satisfied. The boundary condition is also satisfied, following from the continuity of the solution and $x(0;0,\phi)=\mu^{-1}x(T;0,\phi)$.

Finally we prove by contradiction that $\bm{q}(0)\neq0$. If $\bm{q}(0)=0$ then the solution of the differential equation in \eqref{eq:difeq} would be $\bm{q}\equiv0$, which implies that $x(t;0,\phi)\equiv0$ by the monodromy operator action \eqref{eq:nshift}.
However, this contradicts the property that an eigenfunction is not identically zero.

\smallskip

\noindent $\Leftarrow$
We construct function $x(t)$, $t\in[-\tau_h,\, T]$ from $\bm{q}$ in the following way,
\begin{equation}\label{wm:prop2}
\begin{array}{l}
x(t) =  \mu^{a_n}q_{b_n}\left(\nicefrac{t}{\Delta}-(n-1) \right),\ t\in ((n-1)\Delta,\, n\Delta],\ n=-n_h+1,\ldots,N,
\\
x(-\tau_h)=  \mu^{a_{-n_h+1}}q_{b_{-n_h+1}}\left(0\right)
\end{array}
\end{equation}
and show that it is 
a solution initialized at an eigenfunction of $\mathscr{U}$.
Note that \eqref{wm:prop2}  implies
\begin{equation}\label{wm:prop}
x(t)=q_n\left(\nicefrac{t}{\Delta}-(n-1)
\right),\ t\in ((n-1)\Delta,\, n\Delta],\ n=1,\ldots,N
\end{equation}
and, for any $j\in\{1,\ldots,h\}$,
\[
x(t-\tau_j) =  \mu^{a_{n-n_j}}q_{b_{n-n_j}}\left(\nicefrac{t}{\Delta}-(n-1)\right),\ t\in ((n-1)\Delta,\, n\Delta],\, n=-n_h+1+n_j,\ldots,N.
\]
By construction of $x$ and  condition $\bm{q}(1)=B(\mu)\bm{q}(0)$,  $x$ is continuous on $[-\tau_h,\, T]$  and
$
x(t)=\mu^{-1}\;  x(t+T)$, for any  $t\in[-\tau_h,\, 0]$.
It remains to show that $x$ is a solution of \eqref{eq:system}.  In the interval $t\in[(n-1)\Delta,\, n\Delta]$, for any $n= 1,\ldots, N$, we get from \eqref{wm:prop} that
\[
\begin{array}{lcl}
\dot x(t)&=& \nicefrac{1}{\Delta}\; \dot q_n\left(\nicefrac{t}{\Delta}-(n-1)  \right)\\
&=& \sum_{j=1}^h  A_j(t)  \mu^{a_{n-n_j}} q_{b_{n-n_j}} \left(\nicefrac{t}{\Delta}-(n-1)\right) \\
&=& \sum_{j=1}^h  A_j(t) x(t-\tau_j),
\end{array}
\]
which completes the proof.

\section{Arbitrary non-singular coefficient matrix}\label{appendix:mass}
We briefly discuss the generalization to linear time-periodic system of the form
\begin{equation}\label{eq:systemmass}
\textstyle E\dot{x}(t)=\sum\nolimits_{j=0}^hA_j(t,\bK)x(t), 
\end{equation}
where $E\in\R^{d\times d}$ is non-singular, without (explicitly) using its inverse.

The characteristic matrix $\No(\mu;\bK)$ associated  to \eqref{eq:systemmass} is defined by its action on 
$\bm{v}\in\C^{Nd}$ as $\No(\mu;\bK)\bm{v}=\bm{q}(1)-B(\mu)\bm{v}$, with $\bm{q}$ solution of the initial value problem
$$
\begin{cases}
(E\otimes I_N)\, \dot{\bm{q}}(s)=A(s,\mu;\bK)\bm{q}(s), \qquad s\in[0,\,1],\\
\bm{q}(0)=\bm{v}.
\end{cases}
$$
Its ``transposed problem" \eqref{wm:dual}
changes to $ \Mo(\mu;\bK)\bm{v}=(E^\transpose\otimes I_N)\bm{r}(0)-B(\mu)^\transpose\bm{v}$,  with 
$$
\begin{array}{c}
\begin{cases}
(E^\transpose\otimes I_N)\, \dot{\bm{r}}(s)=-A(s,\mu;\bK)^\transpose\bm{r}(s), \qquad s\in[0,\,1],\\
(E^\transpose\otimes I_N)\, \bm{r}(1)=\bm{v}.
\end{cases}
\end{array}
$$
Moreover, the eigenvalues of $ \Mo(\mu;\bK)\bm{v}=0$, are associated to the Floquet multipliers of $ E^\transpose \dot{x}(t)=\sum_{j=0}^hA_j^\transpose (-t+\tau_j,\bK)x(t-\tau_j) $, in the sense of Theorem~\ref{th:dual}  where \eqref{wm:defphat} changes to $\hat{\bm{p}}(s)=R(E^\transpose\otimes I_N) \bm{r}(1-s)$.

\end{document}